\documentclass[12pt]{article}

\usepackage[T1]{fontenc}

\usepackage{amssymb}
\usepackage{graphicx}
\usepackage{epsfig}
\usepackage{here}
\usepackage{psfrag}
\usepackage{latexsym}
\usepackage{epic}
\usepackage{eepic}
\usepackage{amsmath}
\usepackage{enumerate}
\usepackage{appendix}
\usepackage{color}
\usepackage{url}
\usepackage[margin=2cm]{geometry}

\usepackage[affil-it]{authblk}

\usepackage[authoryear]{natbib}
\usepackage{hyperref}
{\providecommand{\noopsort}[1]{}}

\newcommand{\bea}{\begin{eqnarray}}
\newcommand{\eea}{\end{eqnarray}}
\newcommand{\be}{\begin{equation}}
\newcommand{\ee}{\end{equation}}
\newcommand{\beann}{\begin{eqnarray*}}
\newcommand{\eeann}{\end{eqnarray*}}
\newcommand{\bal}{\begin{align}}
\newcommand{\eal}{\end{align}}
\newcommand{\balnn}{\begin{align*}}
\newcommand{\ealnn}{\end{align*}}

\newcommand{\nn}{\nonumber}

\newcommand{\y}{{\bf y}}
\newcommand{\z}{{\bf z}}

\def\R{{\mathbb R}}
\def\ba{\begin{array}}
\def\ea{\end{array}}
\def\bd{\begin{displaymath}}
\def\ed{\end{displaymath}}
\def\Y{{\mathbb Y}}
\def\P{{\mathbb P}}

\def\1{{\mathbf 1}}

\def\A{{\mathbb A}}
\def\E{{\mathbb E}}
\def\F{{\mathbb F}}

\def\M{{\mathbb M}}
\def\N{{\mathbb N}}
\def\G{{\mathbf G}}

\def\X{{\mathbb X}}

\def\G{{\mathbb G}}
\def\T{{\mathbb T}}
\def\TT{{\mathcal T}}

\def\Z{{\mathbb Z}}

\def\B{{\mathcal B}}

\DeclareMathOperator{\e}{e}

\DeclareMathOperator{\Cov}{Cov}
\DeclareMathOperator{\Var}{Var}
\DeclareMathOperator{\argmin}{arg \, min}
\DeclareMathOperator{\supp}{supp}

\usepackage{amsthm}
\newtheorem{thm}{Theorem}
\newtheorem{cor}{Corollary}

\newtheorem{lemma}{Lemma}
\newtheorem{prop}{Proposition}
\newtheorem{rem}{Remark}
\newtheorem{definition}{Definition}

\title{Spatio-temporal c\`adl\`ag functional marked point processes: Unifying spatio-temporal frameworks}
\date{}

\author{Ottmar Cronie%
  \thanks{e-mail: \texttt{ottmar@cwi.nl, ottmar@alumni.chalmers.se} (corresponding author)}}
\affil{
Stochastics research group, CWI\footnote{National Research Institute for Mathematics \& Computer Science}, 
P.O. Box 94079, 
1090 GB Amsterdam, 
The Netherlands}

\author{Jorge Mateu
\thanks{e-mail: \texttt{mateu@mat.uji.es}}
}
\affil{Department of Mathematics, 
Universitat Jaume I, 
Campus Riu Sec, 
12071 Castell\'on, 
Spain}

\begin{document}
\newcommand{\edt}[1]{{\vbox{ \hbox{#1} \vskip-0.3em \hrule}}}

\maketitle

\begin{abstract}

This paper defines the class of c\`adl\`ag functional marked point processes (CFMPPs). These are (spatio-temporal) point processes marked by random elements which take values in a c\`adl\`ag function space, i.e.\ the marks are given by c\`adl\`ag stochastic processes. 
We generalise notions of marked (spatio-temporal) point processes and indicate how this class, in a sensible way, connects the point process framework with the random fields framework. We also show how they can be used to construct a class of spatio-temporal Boolean models, how to construct different classes of these models by choosing specific mark functions, and how c\`adl\`ag functional marked Cox processes have a double connection to random fields. 
We also discuss finite CFMPPs, purely temporally well-defined CFMPPs and Markov CFMPPs.
Furthermore, we define characteristics such as product densities, Palm distributions and conditional intensities, in order to develop statistical inference tools such as likelihood estimation schemes. 

\end{abstract}


\section{Introduction}\label{SectionIntroduction}

Point processes \Citep{CoxIsham1980, DVJ1, DVJ2, Karr1991,VanLieshout, Moller,SKM}, which may be treated as random collections of points falling in some measurable space, have found use in describing an increasing number of naturally arising phenomena, in a wide variety of applications, including epidemiology, ecology, forestry, mining, hydrology, astronomy, ecology, and meteorology
\Citep{CoxIsham1980, DVJ1, Karr1991, Moller, Ripley1981, SchoenbergTranbarger2008, Schoenberg2011, TranbargerSchoenberg2010}.

Point processes evolved naturally from renewal theory and the statistical analysis of life tables, dating back to the 17th century, and in the earliest applications each point represented the occurrence time of an event, such as a death or an incidence of disease (see e.g. \Citep[Chapter 1]{DVJ1} for a review). In the mid-20th century interest expanded to spatial point processes, where each point represents the location of some object or event, such as a tree or a sighting of a species \Citep{Cressie1993, Diggle2003, Ripley1981, SKM}. More recent volumes have a strong emphasis on spatial processes and address mathematical theory \Citep{DVJ2, Handbook,VanLieshout,SchneiderWeil}, methodology of statistical inference \Citep{VanLieshout,Moller}, and data analysis in a range of applied fields \Citep{Diggle2003,Ripley1981,BaddeleyEtAl, Illian}, although the distinction
between these three areas is far from absolute and there are substantial overlaps in coverage
between the cited references.

The classical model for temporal or spatial point processes is the Poisson process, where the number of points in disjoint sets are independent Poisson distributed random variables. Alternative models for spatial point processes
(\Citep[Chapter 8]{Cressie1993} or \Citep{Moller}) grew quite intricate over the course of the 20th century, and among the names associated with these models are some of the key names in the history of statistics, including Jerzy Neyman and David Cox. Today, much attention is paid to spatio-temporal point processes, where each point represents the time and location of an event, such as the origin of an earthquake or wildfire, a lightning strike, or an incidence of a particular disease \Citep{TranbargerSchoenberg2010,VereJones}. 

The intimate relationship between point processes and time series is worth noting. Indeed, many data sets that are traditionally viewed as realisations of point processes could in principle also be regarded as time series, and vice versa \Citep{CoxIsham1980,TranbargerSchoenberg2010}. For instance, a sequence of earthquake origin times is typically viewed as a temporal point process, though one could also store such a sequence as a time series consisting of zeros and ones, with the ones representing earthquakes. The main difference is that for a point process, a point can occur at any time in a continuum, whereas for time series, the time intervals are discretised. In addition, if the points are sufficiently sparse, one can see that it may be far more practical to store and analyse the data as a point process, rather than dealing with a long list containing mostly zeros. By the mid 1990s, models for spatial-temporal point processes had become plentiful and often quite intricate.

A probabilistic view of spatio-temporal processes, in principle, can just regard time as one more
coordinate and, hence, a special case of a higher-dimensional spatial approach. Of course,
this is not appropriate for dynamic spatially referenced processes, as time has a different character than space. There has been a lot of recent work on spatio-temporal models, and a variety of ad hoc approaches have been suggested.
Processes that are both spatially and temporally discrete are more naturally considered
as binary-valued random fields. Processes that are temporally
discrete with only a small number of distinct event-times can be considered initially as
multivariate point processes, but with the qualification that the
temporal structure of the type-label may help the interpretation of any inter-relationships
among the component patterns. Conversely, spatially discrete processes with only a small
number of distinct event-locations can be considered as multivariate temporal point processes,
but with a spatial interpretation to the component processes. The other more common end is considering processes that are temporally continuous and either spatially continuous
or spatially discrete on a sufficiently large support to justify formulating explicitly
spatio-temporal models for the data.

A marked point pattern is one in which each point of the process carries
extra information called a mark, which may be a random variable, several random variables,
a geometrical shape, or some other information. A multivariate or multi-type point pattern is the special
case where the mark is a categorical variable. Marked point patterns with nonnegative real-valued marks are also of interest. A spatial pattern of geometrical objects, such as disks or polygons of different sizes
and shapes, can be treated as a marked point process where the points are the centres of
the objects, and the marks are parameters determining the size and shape of the objects \Citep{RipleySutherland,StoyanStoyan}.

Marked point patterns raise new and interesting questions concerning the appropriate way
to formulate models and pursue analyses for particular applications. In the analysis of a marked point pattern, an important choice is whether to analyse the marks and locations jointly or conditionally. Schematically, if we write $X$ for the points and $M$ for the marks, then we could specify a model for the marked point process
$[X, M]$. Alternatively we may condition on the locations of the points, treating only the marks
as random variables $[M|X]$. In some cases, we may condition on the marks, treating the locations as a random point process $[X|M]$. This is meaningful if the mark variable is a continuous real-valued quantity,
such as time, age or distance. The concept of marking refers to methods of constructing marked point processes
from unmarked ones. Two special cases, independent and geostatistical markings, are among
the known simple examples of marking strategies and are often used in practice. However,
these markings are not able to model density-dependence of marks, the case where the local
point intensity affects the mark distribution.

One important situation is where the marks are provided by a (random) field -- geostatistical/random field marking. A random
field is a quantity $Z(u)$ observable at any spatial location $u$. A typical question is to determine whether $X$ and $Z$ are independent. If $X$ and $Z$ are independent,
then we may condition on the locations and use geostatistical techniques to investigate
properties of $Z$. However, in general, geostatistical techniques, such as the variogram, have
a different interpretation when applied to marked point patterns. In this context, the analysis of dependence between marks and locations is of interest. \Citep{Schlather2004} defined the conditional mean
and conditional variance of the mark attached to a typical random point, given that there
exists another random point at a distance $r$ away from it. These functions may serve as diagnostics for dependence between the points and the marks. Another way to generate non-Poisson marked point processes is to apply dependent thinning to a Poisson marked point process. Interesting examples occur when the thinning rule
depends on both the location and the mark of each point.

Despite the relatively long history of point process theory, few approaches have been considered to analyse
spatial point patterns where the features of interest are functions (i.e. curves) instead of qualitative
or quantitative variables. 
For instance, an explicit example is given by the \emph{growth-interaction process} \citep{Comas,MateuFMPP,Cronie,CronieSarkka,CronieForest,RenshawComas,RenshawComasMateu,RS1,RS2}, which has been used to model the collective development of tree locations and diameters in forest stands. 
Moreover, \Citep{Illian2006} consider for each point a transformed
Ripley's $K$-function to characterise spatial point patterns of ecological plant communities, whilst
\Citep{Mateu2008} build new marked point processes formed by spatial locations and curves defined
in terms of LISA functions, which define local characteristics of the point pattern. They use this
approach to classify and discriminate between points belonging to a clutter and those belonging to
a feature.
The study of such configurations permits to analyse the effects of the spatial structure on individual functions. For instance, the analysis of point patterns where the associated curves depend on time may permit the study of spatio-temporal interdependencies of such dynamic processes.

Functional data analysis describes and models data based on curves \Citep{RamsaySilverman2002,RamsaySilverman2005}. 
This theory considers each curve as an observation rather than a set of numbers \Citep{RamsaySilverman2002}. Therefore, functional data analysis together with point process theory provides
the theoretical framework to analyse point patterns with associated curves. The use of functional
data analysis has already been considered to analyse geostatistical data involving functions instead
of single observations. For instance, \Citep{DelicadoGiraldoComas,GiraldoDelicadoMateu2010,GiraldoDelicadoMateu2011} develop new geostatistical tools to
predict unobserved curves representing daily temperature throughout a year, and analyse a data
set consisting of daily meteorological measurements recorded at several weather stations of Canada.
However, the use of functional tools in point pattern analysis is limited to just a few references and
none of them provides new second order characteristics.

It is clear that there is a wealth of approaches in the theory of spatial point processes. 
However, the large number of derived  spatial point process approaches and methods reduces significantly 
when handling a spatio-temporal structure in combination with such associated marks.
Our aim here is to propose a new class of (spatio-temporal) functional marked point processes, where the marks are random elements which take values in a c\`adl\`ag function space. 
The reason for this choice of function class is its generality and flexibility, and thus its ability to accommodate a variety of different models and structures. 
With this new setup, we generalise most of the usual notions of (spatio-temporal) marked point processes, hence providing a unifying framework. 
In addition, we indicate how this framework in a natural way unifies the frameworks of marked point processes and random fields, and we indicate a geometrical interpretation which connects this framework with (spatio-temporal) Boolean models. 
We develop characteristics such as product densities, Palm distributions and (Papangelou) 
conditional intensities, as these play a significant role in both theoretical as well as practical aspects of point process analysis. 
We also we discuss different explicit marking structures and give a thorough description of the statistical framework when the marks are sampled discretely.

The paper is structured as follows. 
Section \ref{MainSectionSTCFMPP} presents the new class of c\`adl\`ag functional marked point processes, both in its spatial and spatio-temporal versions. Here also some geometric interpretations are discussed. Some motivating examples and connections with other spatio-temporal frameworks are given in Section \ref{SectionExamples}. Section \ref{SectionPointProcessCharacteristics} develops certain point process characteristics, such as product densities and Papangelou conditional intensities, which are needed for the development of the statistical theory underlying these processes. 
Section \ref{SectionMarkStructures} discusses certain specific marking structures, which may be considered within this framework. 
The point process characteristics are particularised to Poisson, Cox, temporally well-defined, finite and Markov c\`adl\`ag functional marked point processes in Section \ref{SectionClassesSTCFMPP}. 
Finally, in Section \ref{SectionDiscretelySampledMarks}, the scenario where the functional marks are sampled discretely is covered. 

\section{C\`adl\`ag functional marked point processes}\label{MainSectionSTCFMPP}

We here describe the construction of two types of point processes, where the second type is a spatio-temporal version of the first type. Heuristically, the first type may be described as a collection $\Psi=\{(X_i,(L_i,M_i))\}_{i=1}^N$ of Euclidean spatial locations $X_i$ with associated function-valued marks $M_i$ 
and auxiliary marks $L_i$, i.e.\ random parameters/variables with the purpose of controlling $M_i$. 
For the second type we further add a random temporal event $T_i$ to each point of $\Psi$ so that the point process $\Psi$ may be described as the collection $\Psi=\{((X_i,T_i),(L_i,M_i))\}_{i=1}^N$. 
Note that the extra parentheses here are meant to emphasise which part is the space-time location and which part is the mark.

We start by defining the two product spaces on which these two types of point processes are defined. 
We then continue to define the two types of point processes $\Psi$ as random measures on these two spaces.

\subsection{Notation}

Let the underlying probability space be denoted by $(\Omega,\mathcal{F},\P)$. Due to the inherent temporally evolving nature of the functional marks and/or the spatio-temporal point process part, at times we will further consider some filtration $\mathcal{F}_{\T}$ and thus obtain a filtered probability space $(\Omega,\mathcal{F},\mathcal{F}_{\T},\P)$. 
We let $\Z_+=\{1,2,\ldots\}$ and $\N=\{0\}\cup\Z_+$, and 
let $\mathcal{P}_N$ denote the power set of $\{1,\ldots,N\}$. 

For any $x,y$ in $d$-dimensional Euclidean space $\R^d$, $d\geq1$, we denote the Euclidean norm by $\|x\|=(\sum_{i=1}^{d}x_i^2 )^{1/2}$ (or sometimes $|x|$) and the Euclidean metric by $d_{\R^d}(x,y)=\|x-y\|$. 
Given some topological space $\mathcal{X}$, we will call $\mathcal{X}$ a \emph{csm} space if it is a complete separable metric space, and as usual we will denote the Borel sets of $\mathcal{X}$ by $\B(\mathcal{X})$. 
Given Borel $\sigma$-algebras $\B(\mathcal{X}_i)$, $i=1,\ldots,n$, we denote the product $\sigma$-algebra by $\bigotimes_{i=1}^{n}\B(\mathcal{X}_i)$ and by $\B(\mathcal{X})^n$ if $\mathcal{X}_i=\mathcal{X}$, $i=1,\ldots,n$. 
For measures $\nu_i(\cdot)$ defined on $\B(\mathcal{X}_i)$, $i=1,\ldots,n$, we write $\bigotimes_{i=1}^n\nu_i(\cdot)$ for the product measure and we write $\nu^n$ if the measure spaces are identical. 
We will denote Lebesgue measure on $(\R^d,\B(\R^d))$ by $\ell$ and use both $\int_{G}f(x)\ell(dx)$ and $\int_{G}f(x)dx$ interchangeably to denote the integral of some measurable function $f:\R^d\rightarrow\R$, with respect to $\ell$ and $G\subseteq\R^d$. 
When we need to emphasise the dimension of the space on which we apply $\ell$, we write e.g.\ $\ell_d$ to denote Lebesgue measure on $\R^d$.

For any set $A$, we let $\1_{A}(a)=\1\{a\in A\}$ denote the indicator function of $A$ and $|A|$ will denote the related cardinality (it will be clear from context whether we consider the norm or the cardinality). 
Given some measurable space $\Y$, we let $\delta_{y}(\cdot)$ denote the Dirac measure of the measurable singleton $\{y\}\subseteq\Y$ and sometimes this notation will also be used for Dirac deltas. 
As usual, a.s.\ will be short for \emph{almost surely} and a.e.\ will be used for \emph{almost everywhere}.

Throughout, by a \emph{kernel} we understand a family $\mu=\{\mu(x,A):x\in\mathcal{X}, A\in\mathcal{F}\}$ such that, for a fixed $x\in\mathcal{X}$, $\mu(x,\cdot)$ is a measure on some $\sigma$-algebra $\mathcal{F}$ and $\mu(\cdot,A)$ is a measurable function for a fixed $A\in\mathcal{F}$. When $\mu$ is a kernel such that each $\mu(x,\cdot)$ is a probability measure on $\mathcal{F}$, we call $\mu$ a family of \emph{regular (conditional) probabilities}.

\subsection{The state spaces}
The spaces $\X$, $\T$, $\A$ and $\F$ below will be used as underlying spaces in the construction of (spatio-temporal) c\`adl\`ag functional marked point processes. 
For instance, as we shall see, a spatio-temporal c\`adl\`ag functional marked point process will be defined as a marked point process with ground space $\X\times\T$ and mark space $\A\times\F$.

\subsubsection{The spatial ground space}
Turning now to the purely spatial domain, throughout we will assume that it is given by some  
subset $\X\subseteq\R^d$, $d\geq1$, with Borel sets $\B(\X)$. 
Hereby, when we construct our point processes, 
each point of the point process will have some spatial location $x\in\X$. 
We note that the most common assumption here is that $\X=\R^d$. 
However, if $\X\neq\R^d$ or $\X\neq\pm[0,\infty)^d$ we will require that $\X$ is compact in order to make it csm (note that this includes the case of identifying the sides of a (hyper)rectangle in order to construct a torus). When this is the case $\Psi$ becomes a finite point process.

\subsubsection{The temporal ground space}
In the case of spatio-temporal point processes we also consider the temporal interval domain $\T\subseteq\R$, which contains a point's (main) temporal occurrence/event time $t\in \T$ (in some applications $t$ symbolises e.g.\ a birth/arrival time). Note that $\T\in\B(\T)\subseteq\B(\R)$ and usually $\T = [0,T^*]\subseteq\{0\}\cup\R_+ = [0,\infty)$.

\subsubsection{The auxiliary mark space}\label{SectionAuxiliaryMarks}
Being marked point process models, at times we need to connect some auxiliary variable to each point of the process. 
Such auxiliary information may possibly represent one of the following things.
\begin{enumerate}
\item A classification of type: Let 
$\A=\A_d=\{1,\ldots,k_{\A}\}$, $k_{\A}\in\Z_+$, i.e.\ each auxiliary mark will be of discrete type. 
Note here that the resulting (spatio-temporal) point process models will be of \emph{multivariate} type \Citep{DVJ1,VanLieshout}. The metric chosen is $d_{\A}(l_1,l_2)=|l_1-l_2|$, $l_1,l_2\in\A$ and the Borel sets are given by $\mathcal{P}_{k_{\A}}$. 

\item 
Continuous auxiliary information: 
Let $\A=\A_c\subseteq\R^{m_{\A}}$ for some $m_{\A}\in\Z_+$ (usually $\A=[0,\infty)$). 
This corresponds to e.g.\ some additional temporal information, such as a \emph{lifetime}, which possibly controls the behaviour of the functional mark. 
Here the metric $d_{\A}(\cdot,\cdot)$ on $\A$ will be given by the Euclidean metric $\|\cdot\|$. 

\item 
The combination of the above: Let $\A=\A_d\times\A_c$, with the metric 
$d_{\A}(l_1,l_2)=\|l_{12}-l_{22}\| + |l_{11}-l_{21}|$, $(l_1,l_2)=((l_{11},l_{12}),(l_{21},l_{22}))\in\A^2$. 
Note that case 2 may be considered superfluous since we here simply may let $k_{\A}=1$, whereby each auxiliary mark will take values in $\{1\}\times\A_c$. 

\end{enumerate}
Note that under each of the proposed metrics, the corresponding space becomes a csm space and we denote the Borel sets by $\B(\A)$.

\subsubsection{The functional mark space}
In a functional marked point process, a (functional) mark may represent an array of things, 
ranging from e.g.\ some feature's growth over time to some function describing spatial dependence. 
In order to accommodate a large range of models and applications, we choose to allow for the functional marks to take values in a Skorohod space (see e.g.\ \Citep{Billingsley,EthierKurtz,JacodShiryaev,Silvestrov}).

More specifically, 
consider some $\TT\subseteq[0,\infty)$, with $T^*=\sup \TT$ (with $T^*=\infty$ if $\TT=[0,\infty)$), and 
consider the function space
\beann
\F=D_{\TT}(\R) = \{f:\TT\rightarrow\R | f\text{ c\`adl\`ag}\},
\eeann
which is the set of c\`adl\`ag (right continuous with existing left limits) functions $f:\TT\rightarrow(\R,d_{\R}(\cdot,\cdot))$ (see e.g.\ \Citep{Billingsley}). 
Consider now the collection $\Lambda$ of all strictly increasing, surjective and Lipschitz continuous functions $\lambda:\TT\rightarrow\TT$, $\lambda(0)=0$, $\lim_{t\rightarrow\infty}\lambda(t)=T^*$, such that
$$
\gamma(\lambda) = \sup_{s,t\in\TT : t<s}\left|\log\frac{\lambda(s)-\lambda(t)}{s-t}\right| < \infty.
$$
Since $(\R,d_{\R}(\cdot,\cdot))$ is a csm space, by endowing $\F$ with the metric
\beann
d_{\F}(f,g) = \inf_{\lambda\in\Lambda}
\left\{
\gamma(\lambda) \vee
\int_{\TT} \e^{-u} \sup_{t\in\TT}\{
d_{\R}(
f(t\wedge u),
g(\lambda(t)\wedge u)
)
\wedge1
\}
du
\right\},
\eeann
we turn it into a csm space \Citep{EthierKurtz}. 
The Borel sets generated by the corresponding topology will be denoted by $\B(\F)$ and it follows that $\B(\F^n)=\B(\F)^n$ \Citep{JacodShiryaev}. 
Consider now the following definition, given in accordance with \Citep[1.6.1]{Silvestrov}.
\begin{definition}\label{DefCadlagProcess}
A stochastic process $X(t)=(X_1(t),\ldots,X_n(t))$, $n\geq1$, $t\in\TT$, is called an $n$-dimensional \emph{c\`adl\`ag stochastic process} if each of its sample paths $X(\omega)=\{X(t;\omega)\}_{t\in\TT}$, $\omega\in\Omega$, is an element of $\F^n$.
\end{definition}
In light of this definition, we note that functions in $\F$ include e.g.\ sample paths of Markov processes, L\'{e}vy processes and semi-martingales, as well as empirical distribution functions. 
We further note that the space $C_{\TT}(\R) = \{f:\TT\rightarrow\R : f \text{ continuous}\}$ is a subspace of $\F$ and for these functions $d_{\F}$ reduces to the uniform metric $d_{\infty}(f,g)=\sup_{t\in\TT}|f(t)-g(t)|$. 
In addition, the Borel $\sigma$-algebra $\B(C_{\TT}(\R))$ generated by $d_{\infty}(\cdot,\cdot)$ on $C_{\TT}(\R)$ satisfies $\B(C_{\TT}(\R))=\{E\cap C_{\TT}(\R):E\in\B(\F)\}\subseteq\B(\F)$ \Citep[Chapter VI]{JacodShiryaev}. For details on filtrations with respect to c\`adl\`ag stochastic processes, see \Citep[Chapter VI]{JacodShiryaev}. 
Hence, we can accommodate e.g.\ diffusion processes or some other class of processes with continuous sample paths (note also that each space $C_{\TT}^k(\R)$, $k\in\N$, of $k$ times continuously differentiable functions is a subspace of $C_{\TT}(\R)$).

\subsection{The spatial and spatio-temporal state spaces}
Since both $\A$ and $\F$ are csm, 
by endowing $\M=\A\times\F$ with the supremum metric
$$
d_{\M}((l_1,f_1),(l_2,f_2)) = \max \{d_{\A}(l_1,l_2),d_{\F}(f_1,f_2)\},
\quad (l_1,f_1),(l_2,f_2)\in\M,
$$
(or any other equivalent metric) $\M$ itself becomes csm \Citep[p.\ 377]{DVJ1} and its Borel sets are given by
$\B(\M)=\B(\A\times\F)=\B(\A)\otimes\B(\F)$ (see e.g.\ \Citep[Lemma 6.4.2.]{Bogachev}). 

\subsubsection{The spatio-temporal state space}
Let $\G=\X\times\T$ and endow it with the supremum norm 
$\|(x,t)\|_{\infty} = \max\{\|x\|,|t|\}$ and the supremum metric 
\[
d_{\G}((x_1,t_1),(x_2,t_2)) = \|(x_1,t_1)-(x_2,t_2)\|_{\infty} =\max\{d_{\R^d}(x_1,x_2),d_{\R}(t_1,t_2)\},
\] 
where $(x_1,t_1),(x_2,t_2)\in\G$, 
so that $\G$ becomes a csm space and $\B(\G)=\B(\X\times\T) = \B(\X)\otimes\B(\T)$. 
We note that there are other possible equivalent metrics, which measure space and time differently (this is needed since it is the defining property of spatio-temporal point processes). However, for our purposes, this is the preferable choice. 

\begin{rem}
\label{RemarkMetricScaling}
Note that if the scales of time and space need to be altered, 
we may rescale e.g.\ time by letting $\|(x,t)\|_{\infty} = \max\{\|x\|,\beta|t|\}$, $\beta>0$. 
The current construction amounts to $\beta=1$. 
\end{rem}

The resulting underlying measurable spatio-temporal space which we will consider is given by 
$$
(\Y,\B(\Y)) = (\G\times\M,\B(\G\times\M)) =((\X\times\T)\times(\A\times\F),\B(\X)\otimes\B(\T)\otimes\B(\A)\otimes\B(\F))
$$
and we note that $\Y$ is a Polish space, as a product of Polish spaces. In fact, by endowing $\Y=\G\times\M$ with the supremum metric 
$$
d((x_1,t_1,l_1,f_1),(x_2,t_2,l_2,f_2))=\max\{d_{\G}((x_1,t_1),(x_2,t_2)),d_{\M}((l_1,f_1),(l_2,f_2))\},
$$ 
$\Y$ itself becomes a csm space \Citep[p.\ 8]{VanLieshout}.

Concerning $\G$, 
for any $(x,t)\in\G$ and $u,v\geq0$, consider the cylinder set
\bea
\label{CylinderSet}
(x,t) +C_u^v
&=&
(x,t) + \{(y,s)\in\G: \|y\|\leq u,|s|\leq v\}\\ 
&=& \{(y,s)\in\G: d_{\R^d}(x,y)\leq u, d_{\R}(t,s)\}\leq v\}.\nn
\eea
We see that in this metric space closed balls satisfy $B[(x,t),u] = (x,t) + C_u^u$. 

\begin{rem}
We note that in many, if not most, cases it is desirable to set $\TT=\T$ so that $\T$ describes the total part of time which we are considering for the constructed point process on $\Y$. 
\end{rem}

\subsubsection{The spatial state space}
The same reasoning gives us the (explicitly) non-temporal space 
$$
(\Y,\B(\Y)) = (\G\times\M,\B(\G)\otimes\B(\M)) =(\X\times(\A\times\F),\B(\X)\otimes\B(\A)\otimes\B(\F))
$$ 
with $\G$ having underlying norm $\|\cdot\|$ and metric 
$d_{\G}(x_1,x_2) = d_{\R^d}(x_1,x_2)$. 
We see here that the only temporal information present is found implicitly in each $f=\{f(t):t\in \TT\}\in\F$, provided that $t\in\TT$ describes time.

\subsection{Reference measures and reference c\`adl\`ag stochastic processes}
\label{SectionRefernceMeasures}
When constructing marked point processes, for various reasons, including the derivation of explicit structures for different summary statistics, one has to choose a sensible reference measure $\nu_{\M}$ for the mark space $(\M,\B(\M))$. For similar reasons one also usually considers some reference measure $\nu_{\G}$ on the ground space $(\G,\B(\G))$. 
We here let the reference measure on $(\Y,\B(\Y))$ be given by
\bea
\label{ReferenceMeasure}
\nu(\cdot) = [\nu_{\G}\otimes\nu_{\M}](\cdot)
= [\ell\otimes[\nu_{\A}\otimes\nu_{\F}]](\cdot)
,
\eea
where each component measure in $\nu$ governs the probabilistic structures of $\Psi$ on $\G$, $\A$ and $\F$, respectively. 

Regarding the measure on $\G$, we let it be given by Lebesgue measure $\ell$, where $\ell=\ell_{d}$ if $\G=\X$ and $\ell=\ell_{d+1}=\ell_{d}\otimes\ell_1$ if $\G=\X\times\T$. 
This is the usual choice when constructing point processes on $\R^d$ or spatio-temporal point processes on $\R^d\times\R$ (recall that the metrics are different). 

Recall the different auxiliary mark spaces given in Section \ref{SectionAuxiliaryMarks}. 
Irrespective of whether $\A=\A_d$, $\A=\A_c$ or $\A=\A_d\times\A_c$, we let the auxiliary mark reference measure $\nu_{\A}$ be given by some locally finite Borel measure on $\B(\A)$, i.e.\ $\nu_{\A}(D)<\infty$ for bounded $D\in\B(\A)$. 
In Section \ref{SectionAuxiliaryMarkMeasure} we discuss in detail some possible choices for $\nu_{\A}$.

Turning to the functional mark space $(\F,\B(\F))$, 
consider some suitable reference c\`adl\`ag stochastic process
\begin{align}
\label{ReferenceProcess}
&X_{\F}:(\Omega,\mathcal{F},\P)\rightarrow(\F,\B(\F)),
\\
&\Omega\ni\omega\mapsto X_{\F}(\omega) = \{X_{\F}(t;\omega)\}_{t\in\TT}\in\F,
\nn
\end{align}
where each $X_{\F}(\omega)$ is commonly referred to as a sample path/realisation of $X_{\F}$, 
and consider the induced probability measure 
\beann
\nu_{\F}(E) = \P(\{\omega\in\Omega : X_{\F}(\omega)\in E\}),
\quad E\in\B(\F),
\eeann
which will be the canonical reference measure under consideration. 
Note that the joint distribution on $(\F^n,\B(\F^n))$ of $n$ independent copies of $X_{\F}$ is given by $\nu_{\F}^n$, the $n$-fold product measure of $\nu_{\F}$ with itself. 
Also, we may conversely first choose the measure $\nu_{\F}$ and then consider the corresponding process $X_{\F}$.

For reasons which will become clear, $\nu_{\F}$ or $X_{\F}$ should be chosen so that suitable absolute continuity/change-of-measure results can be applied. 
More specifically, the distribution $P_X$ on $(\F^n,\B(\F^n))$, $n\geq1$, of some stochastic process $X=\{X(t)\}_{t\in\TT}\in\F^n$ of interest should have some (functional) Radon-Nikodym derivative $f_X$ with respect to $\nu_{\F}^n$, i.e.\ 
$P_X(E)=\int_{E}f_X(f)\nu_{\F}^n(df)=\E_{\nu_{\F}^n}[\1_E f_X]$, $E\in\B(\F^n)$ (see \Citep{Skorohod} for a discussion on such densities). 
In Section \ref{SectionExplicitFunctionalMeasures} we discuss such choices further and 
we look closer at Wiener measure as reference measure, i.e.\ the measure induced by a Brownian motion $X_{\F}=W=\{W(t)\}_{t\in\TT}$.

\subsection{Point processes}

Having defined the state spaces for the two types of point processes defined here, we now turn to their actual definitions. 

Let $(\Y,\B(\Y))$ be given by any of the two state spaces defined above. 
Furthermore, let $\mathcal{N}_{\Y}$ be the  collection of all locally finite counting measures $\varphi=\sum_{y\in\varphi}\delta_{y}$ on $\B(\Y)$, i.e.\ $\varphi(A)<\infty$ for bounded $A\in\B(\Y)$ and denote the corresponding counting measure $\sigma$-algebra by $\Sigma_{\mathcal{N}_{\Y}}$ (see \Citep[Chapter 9]{DVJ2}). 
Note that in what follows we will not distinguish in the notation between a measure $\varphi\in\mathcal{N}_{\Y}$ and its support $\varphi\subseteq\Y$ whereby $\varphi(\{y\})>0$ and $y\in\varphi$ (or $|\varphi\cap\{y\}|\neq0$) will mean the same thing for any $y\in\Y$.

\begin{definition}
If $\Psi:\Omega\rightarrow\mathcal{N}_{\Y}$, $\omega\mapsto\Psi(\cdot;\omega)$, is a measurable mapping from the probability space $(\Omega,\mathcal{F},\P)$ into the space $(\mathcal{N}_{\Y},\Sigma_{\mathcal{N}_{\Y}})$, we call $\Psi$ a \emph{point process} on $\Y$. 
\end{definition}
Denote further by $\mathcal{N}_{\Y}^*$ the sub-collection of $\varphi\in\mathcal{N}_{\Y}$ 
such that the \emph{ground measure} $\varphi_{G}(\cdot)=\varphi(\cdot\times\M)$ is a locally finite simple counting measure on $\B(\G)$ 
(simple means that $\varphi_{G}(\{g\})\in\{0,1\}$ for any $g\in\G$). 
We note that the simplicity of the ground measure further implies that $\varphi(\{(g,m)\})\leq\varphi_{G}(\{g\})\in\{0,1\}$ for any $(g,m)\in\G\times\M$.

Throughout, irrespective of the choice of $\G$, for any $\varphi=\sum_{(g,l,f)\in\varphi}\delta_{(g,l,f)}\in \mathcal{N}_{\Y}$ (where $g\in\G$ and $(l,f)\in\M$) we will write 
$\varphi + z = \sum_{(g,l,f)\in\varphi}\delta_{(g+z,l,f)}$ to denote a shift of $\varphi$ in the ground space by the vector $z\in\G$. This notation will, in particular, be used in the definition of stationarity. 

Recalling Definition \ref{DefCadlagProcess}, 
we see that any collection of elements $\{(g_1,l_1,f_1),\ldots,(g_n,l_n,f_n)\}\subseteq\Psi$ consists of the combination of a) a collection of spatial(-temporal) points $g_1,\ldots,g_n\in\G$, b) a collection $l_1,\ldots,l_n$ of random variables taking values in $\A$, and c) an $n$-dimensional c\`adl\`ag stochastic process $(f_1(t),\ldots,f_n(t))$, $t\in\TT$, all tied together.

\subsection{C\`adl\`ag functional marked point processes}\label{SectionCFMPP}

Following the terminology and structure given in \Citep{MateuFMPP}, we now have the following definition.

\begin{definition}\label{DefCFMPP}
Let $\Y=\G\times\M=\X\times(\A\times\F)$ and let $\Psi:\Omega\rightarrow\mathcal{N}_{\Y}$ be a point process on $\Y = \X\times(\A\times\F)$. If $\Psi\in\mathcal{N}_{\Y}^*$ a.s., we call 
$$
\Psi 
= \sum_{y\in\Psi}\delta_{y} 
= \sum_{(x,l,f)\in\Psi}\delta_{(x,l,f)}
$$
a (simple) \emph{c\`adl\`ag functional marked point process} (CFMPP) on $\Y$. 
\begin{itemize}
\item
If either $\A=\A_d$ or $\A=\A_d\times\A_c$, with $k_{\A}\geq2$ different type classifications, we call $\Psi$ a \emph{multivariate CFMPP}. 

\item
If further $\Psi$ a.s.\ takes its values in  
$\mathcal{N}^f=\{\varphi\in\mathcal{N}_{\Y}^*: \varphi(\Y)<\infty\}\subseteq\mathcal{N}_{\Y}$, we call it a \emph{finite} CFMPP. 
\end{itemize}

\end{definition}

We note that through a unique measurable enumeration (see \Citep[Chapter 9.1]{DVJ2}), we may write  
$$
\Psi
= \sum_{i=1}^{N}\delta_{(X_i,L_i,M_i)}, 
\quad 0\leq N\equiv\Psi(\Y)\leq\infty,
$$
for some sequence $\{(X_i,L_i,M_i)\}_{i=1}^{N}$ of random vectors, which geometrically corresponds to the support of $\Psi$. 
Here $X_i\in\R^d$ represents the spatial location of the $i$th point, $L_i\in\A$ its auxiliary mark and $M_i\in\F$ its functional mark. 
Note that when $\Psi$ is multivariate and $\A=\A_d\times\A_c$, to emphasise this aspect we often write $L_i=(L_{i1},L_{i2})$.
It should further be noted that by construction the \emph{ground process} (unmarked process) 
$$
\Psi_{\X}(\cdot) = \Psi_{G}(\cdot) = \sum_{x\in\Psi_G}\delta_x(\cdot) 
=\sum_{y\in\Psi}\delta_y(\cdot\times\A\times\F), 
$$
with support $\Psi_{G}=\{X_i\}_{i=1}^{N}\subseteq\X$, 
is a well-defined simple point process on $\X$ with 
$\Psi_{G}(B) = |\Psi_{G}\cap B| = \Psi(B\times\A\times\F) 
<\infty$ a.s.\ for bounded $B\in\B(\X)$. 
Note that the dual notation $\Psi_{G}=\Psi_{\X}$ is introduced for later convenience. 

\begin{rem}
Implicitly in the definition of a CFMPP we assume that $\Psi$ is simple (since $\Psi\in\mathcal{N}_{\Y}^*$ a.s.). 
Furthermore, 
if $\Psi_{\F}(\cdot) := \Psi(\X\times\cdot)$ is locally finite, then $\Psi_{\F}$ becomes a well-defined point process on $\F$ and we refer to $\Psi_{\F}$ as the associated mark space point process. 
However, we will not necessarily make that assumption here. 

\end{rem}

As already noted, by construction the collection of marks $\Psi_{\F}=\{M_i\}_{i=1}^{N}$, $M_i=\{M_i(t)\}_{t\in \TT}$, consists of random elements in $\F$ (functional random variables), which simply are c\`adl\`ag stochastic processes with sample paths/realisations $M_i(\omega)=\{M_i(t;\omega)\}_{t\in \TT}\in\F$, $\omega\in\Omega$. As we will see, by letting $M_i$ be given by a point mass $\delta_{f}$ on $(\F,\B(\F))$ we also have the possibility to consider marks which are given by deterministic functions $f$.

\subsection{Spatio-temporal c\`adl\`ag functional marked point processes}\label{SectionSTCFMPP}

We now turn to the case where we include the explicit temporal space $\T$ and, consequently, deal with spatio-temporal CFMPPs. 
Recall that $\Y = \G\times\M = (\X\times \T)\times(\A\times\F)$. 

\begin{definition}\label{DefSTCFMPP}
Let $\Psi:\Omega\rightarrow\mathcal{N}_{\Y}$ be a point process on $\Y = \G\times\M = (\X\times \T)\times(\A\times\F)$. If $\Psi\in\mathcal{N}_{\Y}^*$ a.s., we call 
$$
\Psi 
= \sum_{y\in\Psi}\delta_{y} 
= \sum_{(x,t,l,f)\in\Psi}\delta_{(x,t,l,f)}
$$
a (simple) \emph{spatio-temporal c\`adl\`ag functional marked point process} (STCFMPP) on $\Y$.
\begin{itemize}
\item
If either $\A=\A_d$ or $\A=\A_d\times\A_c$, with $k_{\A}\geq2$ different type classifications, we call $\Psi$ a \emph{multivariate STCFMPP}. 

\item
When $\Psi\in\mathcal{N}^f=\{\varphi\in\mathcal{N}_{\Y}^*: \varphi(\Y)<\infty\}\subseteq\mathcal{N}_{\Y}$ a.s., we call $\Psi$ a \emph{finite} STCFMPP. 
\end{itemize}

\end{definition}

A few things should be mentioned at this point. 
To begin with we note that we may write 
$$
\Psi 
= \sum_{i=1}^{N}\delta_{(X_i,T_i,L_i,M_i)}, 
\quad 0\leq N\equiv\Psi(\Y)\leq\infty,
$$ 
where all $X_i\in\R^d$ represent the spatial locations, $T_i\in \T$ the occurrence times, $L_i\in \A$ the related auxiliary marks and $M_i=\{M_i(t)\}_{t\in \TT}\in\F$ the functional marks. 
In connection hereto, an interesting feature which sets this scenario apart from the non-spatio-temporal CFMPP case is that we here have a natural enumeration/order of the points, which is obtained by assigning the indices $1,\ldots,N$ to the points according to their ascending occurrence times 
$T_1<\ldots<T_N$. 
Hereby the support may be written as $\Psi=\{((X_i,T_i),(L_i,M_i))\}_{i=1}^{N}=\{(X_i,T_i,L_i,M_i)\}_{i=1}^{N}$. 
Also here, when $\Psi$ is multivariate and $\A=\A_d\times\A_c$, we sometimes write $L_i=(L_{i1},L_{i2})$.

We note further that, by construction, the ground process is a well-defined simple point process on $\X\times \T$, i.e.\ 
$$
\Psi_{\X\times \T}(B\times C)=
\Psi_{G}(B\times C) 
= \sum_{(x,t)\in\Psi_G}\delta_{(x,t)}(B\times C) 
= \Psi(B\times C\times(\A\times\F))< \infty,
$$ 
for bounded $B\times C\in\B(\X\times \T)$, with support $\Psi_{G}=\{(X_i,T_i)\}_{i=1}^{N}$. However, at times it may be useful to require that also $\Psi_{\X}=\{X_i\}_{i=1}^{N}$ and/or $\Psi_{\T}=\{T_i\}_{i=1}^{N}$ constitute well-defined point processes. 

\begin{definition}
\label{DefinitionGrounding}
Let $\Psi$ be a STCFMPP. 
\begin{itemize}
\item
If $\Psi_{\X}(\cdot)=\Psi_{G}(\cdot\times\T) = \Psi(\cdot\times \T\times\A\times\F)$ is simple and locally finite, i.e.\ 
the spatial part $\Psi_{\X}$ of the ground process also constitutes a well-defined simple point process on $\X$, 
we say that $\Psi$ is \emph{spatially grounded}.

\item
Similarly, if $\Psi_{\T}(\cdot)=\Psi_{G}(\X\times\cdot) = \Psi(\X\times\cdot\times\A\times\F)$ is simple and locally finite, so that $\Psi_{\T}$ constitutes a well-defined point process on $\T$, we say that $\Psi$ is \emph{temporally grounded}.
\end{itemize}
\end{definition}

To additionally ground $\Psi$ spatially and/or temporally can be of importance for different reasons. 
For instance, as we shall see, we may speak of three different types of stationarity of $\Psi$ and 
when $\Psi$ is temporally grounded we may e.g.\ define conditional intensities, as considered by e.g.\ \Citep{Ogata,SchoenbergIntensity,VereJones}.

\subsection{The point process distribution}
In what follows, we only distinguish in the notation between CFMPPs and STCFMPPs when necessary. 
Let $\Psi$ be a (ST)CFMPP and let 
the induced probability measure of $\Psi$ on $\Sigma_{\mathcal{N}_{\Y}}$ be denoted by $P$, i.e.
\begin{align*}
\P(\{\omega\in\Omega:\Psi(\omega)\in R\}) 
=
P(R)
= \int_{R} P(d\varphi),
\quad 
R\in\Sigma_{\mathcal{N}_{\Y}}.
\end{align*}
Note that since $\Psi$ is a simple point process on the csm space $\Y$, $P$ is completely and uniquely determined by its finite dimensional distributions, i.e.\ the collection of joint distributions of $(\Psi(A_1),\ldots,\Psi(A_n))$ for all collections of bounded $A_i\in\B(\Y)$, $i=1,\ldots,n$, $n\in\Z_+$, as well as by its void probabilities $v(A)=\P(\Psi(A)=0)$, $A\in\B(\Y)$ (see e.g.\ \Citep[Chapter 1]{VanLieshout}).

\subsection{Stationarity and isotropy}

We next give the definition of stationarity which, irrespective of the choice of $\G$, is the usual definition for marked point processes, i.e.\ translational invariance of the ground process. 
\begin{definition}
Let $\Psi$ be a (ST)CFMPP.
\begin{itemize}
\item
Then $\Psi$ is \emph{stationary} if $\Psi + z \stackrel{d}{=} \Psi$ for any $z\in\G$. 
\item
In the case of a STCFMPP, for $z=(a,b)\in\X\times\T$, if stationarity only holds when $b=0$ we say that $\Psi$ is \emph{spatially stationary} and if it only holds when $a=0$ we say that $\Psi$ is \emph{temporally stationary}.
\item
$\Psi$ is \emph{isotropic} if it is stationary and, in addition, $\Psi_{G}$ is rotation invariant with respect to rotations about the origin $0\in\G$.
\end{itemize}
\end{definition}
We see e.g.\ that when $\Psi$ is temporally grounded, temporal stationarity implies that $\Psi_{\T}$ is a stationary point process on $\T$. 
Note further that one often refers to $\Psi$ as \emph{homogeneous} if it is both stationary and isotropic, and as \emph{inhomogeneous} if it is not stationary.

\subsection{Supports}
Being a stochastic process, 
which may have zeroes on $\TT$, 
conditionally on $\Psi_G$ and the auxiliary marks $L_i\in\A$, we may consider different types of supports for each functional mark $M_i=\{M_i(t)\}_{t\in \TT}$. 
We distinguish between the \emph{deterministic support} $\mathrm{supp}(M_i)=\{t\in\TT:M_i(t;\omega)\neq0 \text{ for any }\omega\in\Omega\}$ and the \emph{stochastic support} $\mathrm{supp}^*(M_i)=\{t\in\TT:M_i(t)\neq0\}$. 
We note that $\mathrm{supp}^*(M_i)\subseteq\mathrm{supp}(M_i)$ is a random subset of $\TT$ 
and, moreover, 
$$
\mathrm{supp}(M_i) 
= \bigcup_{\omega\in\Omega} \{t\in\TT:M_i(t;\omega)\neq0\}
= \bigcup_{\omega\in\Omega} \mathrm{supp}^*(M_i;\omega) 
.
$$
For a STCFMPP $\Psi$, when $L_i\geq0$ represents a (random) time and we condition on the $T_i$'s and the $L_i$'s, a natural construction would be to let $D_i=(T_i+L_i)\wedge T^*$ and $\mathrm{supp}(M_i)=[T_i, D_i)$ so that $M_i(t)=0$ for all $t\notin[T_i, D_i)$. 
The interpretation here would be that $T_i$ symbolises the \emph{birth time} of the $i$th point, $L_i$ its \emph{lifetime} and $D_i$ its \emph{death time}. 
Furthermore, we note that if e.g.\  $M_i(t)=\1_{[T_i,D_i)}(t)W_i((t-T_i)\wedge0)$ for a Brownian motion $W_i$, then $\ell_{1}((\mathrm{supp}(M_i)\setminus\mathrm{supp}^*(M_i))^c)=\ell_1(\{t\in[0,L_i):W_i(t)=0\})=0$ a.s.\ (see e.g.\ \Citep{Klebaner}). 
In the case of a CFMPP one could similarly let $\mathrm{supp}(M_i)=[0, L_i)$.

\subsection{Geometric representation and spatio-temporal Boolean models}
Let $\Psi$ be a STCFMPP where $\X\subseteq\R^2$ and where each $M_i$ a.s.\ takes values in the sub-space of continuous functions on $\TT$. 
One possibility for interpretation is obtained by letting the disk (ball) $B_{\X}[X_i,M_i(t)]=\{x\in\R^2:d_{\R^2}(X_i,x)\leq M_i(t)\}$ with centre $X_i$ and radius $M_i(t)$ illustrate the space occupied by the $i$th point of $\Psi$ at time $t\in\T$ (with the convention that $B_{\X}[X_i,r]=\emptyset$ if $r\leq0$). Hereby, at time $t$ we may illustrate $\Phi_M$ as the Boolean model (see e.g.\ \Citep{SKM}) 
$$
\Xi(t) = \bigcup_{i=1}^{N} B_{\X}[X_i,M_i(t)] = \bigcup_{(X_i,T_i,L_i,M_i)\in\Psi:t\in\supp^*(M_i)} B_{\X}[X_i,M_i(t)].
$$
Consequently, $\Psi$ may be represented by the collection 
$$
\Xi = \int_{\T} \Xi(dt)
= \bigcup_{i=1}^{N} \Xi_i = \bigcup_{i=1}^{N} 
\{(x,y,z)\in\R^3 : z\in\supp^*(M_i), d_{\R^2}(X_i,(x,y))\leq M_i(z)\}
$$
and we see that whenever $\supp(M_i)$ is bounded, each \emph{deformed cone} $\Xi_i$ a.s.\ is a compact subset of $\R^3$. Figure \ref{FigureGeometry} illustrates a realisation of such a random set $\Xi$. 
Hence, the cross section of $\Xi$ at $z=t$ gives us $\Xi(t)$ and 
in the context of e.g.\ forest stand modelling, we find that $\Xi(t)$ gives us the geometric representation of the cross section of the forest stand at time $t$, at some given height (usually \emph{breast height}). 
Note that when in addition $\ell(\X)<\infty$, depending on the form of the functional marks, we may derive geometric properties such as the expected coverage proportion $\frac{\pi}{\ell(\X)}\sum_{n=0}^{\infty}\sum_{i=1}^{n}\E[M_i(t)^2]\P(N=n)$ of $\X$ at time $t$ (provided that the disks do not overlap). 

\begin{figure}[htbp]
\begin{center}
\includegraphics[width=0.5\textwidth]{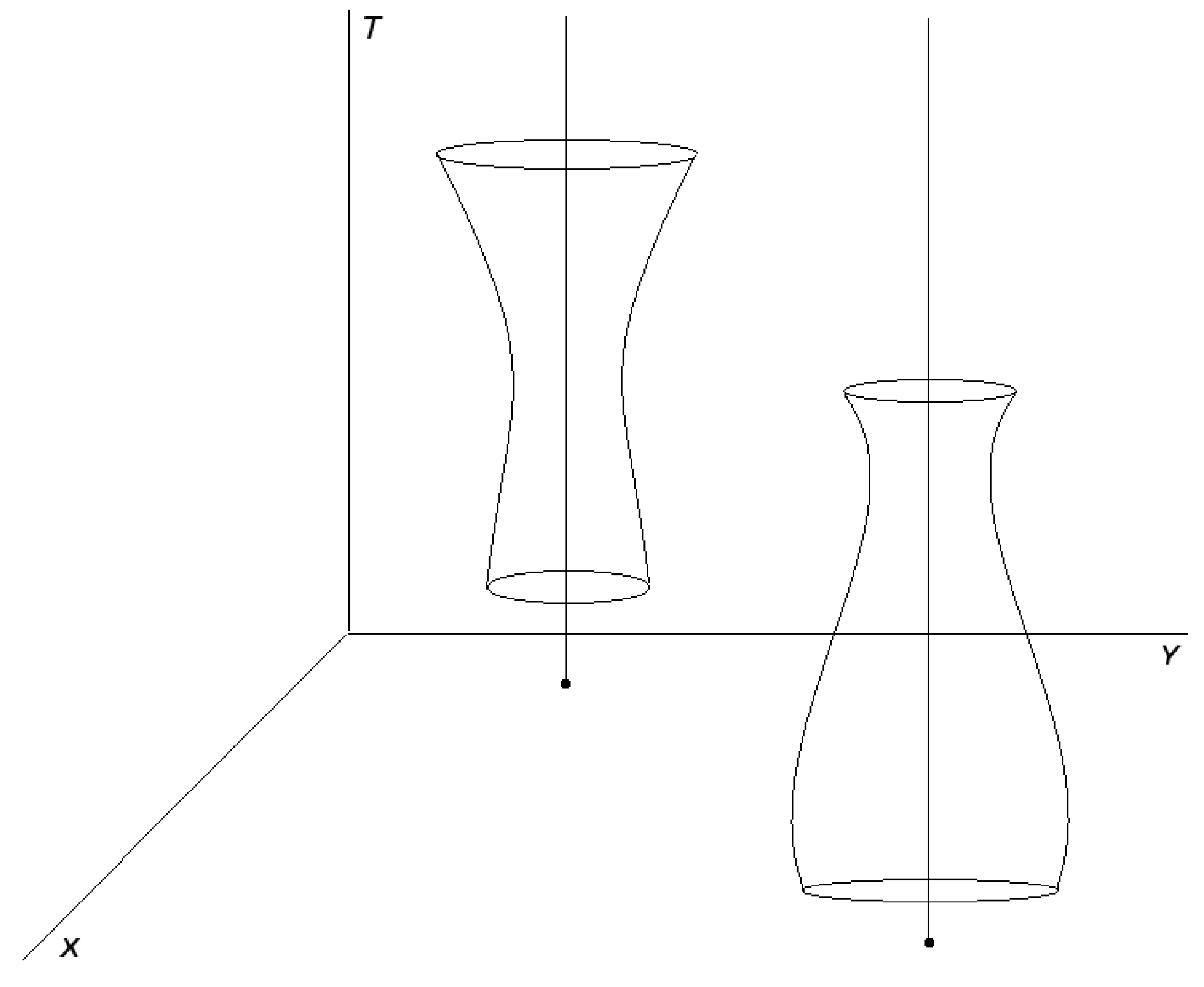}
\caption{A realisation of a random set $\Xi$.}
\label{FigureGeometry}
\end{center}
\end{figure}

\section{Examples of (ST)CFMPPs}\label{SectionExamples}
The class of (ST)CFMPPs provides a framework to give structure to a series of existing models and it allows for the construction of new important models and modelling frameworks, which have uses in different applications. 
Below we present a few such instances. 

\subsection{Marked (spatio-temporal) point processes}\label{SectionMPP}
Not surprisingly, (ST)CFMPPs generalise ordinary marked (spatio-temporal) point processes. 
To see this, by considering the class $\F_c=\{f\in\F : f \text{ is constant}\}$, we find that when $\Psi$ a.s.\ is restricted to the space $\G\times(\A\times\F_c)$, its functional marks become constant functions $M_i(t) = \xi_i$ with (random) values $\xi_i\in\R$, $i=1,\ldots N$. 
Hereby, in the case of a CFMPP $\Psi$, $\bar{\Psi} = \{(X_i,L_i,M_i(0)):X_i\in\Psi_{G}\} = \{(X_i,L_i,\xi_i):X_i\in\Psi_{G}\}$ gives us a classical definition of a marked point process (provided $0\in\supp(M_i)$). 
Similarly, in the case of a STCFMPP $\Psi$, $\bar{\Psi} = \{(X_i,T_i,L_i,M_i(0)):X_i\in\Psi_{G}\} = \{(X_i,T_i,L_i,\xi_i):X_i\in\Psi_{G}\}$ gives the classical definition of a marked spatio-temporal point process. 

Here, when $\Psi$ is multivariate with, say, $\A=\A_c$, $L_i$ would be a discrete variable which describes a point's type and $\xi_i$ would be the size of the quantitative mark. As a consequence $\bar{\Psi}$ would become multivariate.

A slightly more direct way of creating, say, a marked spatio-temporal point process $\bar{\Psi}$ through a STCFMPP is to let 
$$
\bar{\Psi}(B\times C\times D)=\Psi(B\times C\times D\times\F)
= \sum_{(x,t,l,f)\in\Psi}\delta_{(x,t,l,f)}(B\times C\times D\times\F),
$$
$B\times C\times D\in\B(\X)\times\B(\T)\times\B(\A)$
so that the $i$th mark is given by $L_i\in\A$. 
This is naturally possible only if $\bar{\Psi}$ constitutes a well-defined point process in its own right.

\subsection{Spatio-temporal geostatistical marking and geostatistics with uncertainty in the sampling locations}
\label{SectionGeostatMarkin}

For classic marked point processes $\bar{\Psi}=\{(X_i,M_i)\}_{i=1}^{N}$ one often speaks of \emph{geostatistical marking} \Citep{Illian}. This is the case where, conditionally on $X_i$, the marks $M_i=Z_{X_i}$, $i=1,\ldots,N$, are provided by some random field $Z=\{Z_x\}_{x\in\X}$. This may be regarded as \emph{sampling the random field $Z$ at random locations}, provided by $\{X_i\}_{i=1}^{N}$. 
Within the CFMPP-context this idea may further be extended to the case of marks coming from a spatio-temporal random field $Z=\{Z_{x}(t)\}_{(x,t)\in\X\times\TT}$. 
\begin{definition}\label{DefGeostatMarking}
Consider a spatio-temporal random field $Z=\{Z_{x}(t)\}_{(x,t)\in\X\times\TT}$. 
If, conditionally on $\Psi_G$ and $\{L_i\}_{i=1}^{N}$, the marks of a (ST)CFMPP $\Psi$ are given by $M_i = \{Z_{X_i}(t)\}_{t\in\TT}$, $i=1,\ldots,N$, we say that $\Psi$ has a \emph{spatio-temporal geostatistical marking}. 
\end{definition}
We may also refer to this type of marking as 
\emph{sampling a spatio-temporal random field at random spatial locations}. 

Given (dependent) random fields $Z_j=\{Z_j(x,t)\}_{(x,t)\in\X\times\TT}$, $j=1,\ldots,k$, when $\Psi$ is multivariate, natural constructions include 
\begin{itemize}
\item
$M_i(t) = \sum_{j=1}^{k_{\A}}\1\{L_{i}=j\}Z_{j}(X_i,t)$, when $\A=\A_d$,
\item
$M_i(t) = \1_{[T_i,T_i+L_{i2})}(t)\sum_{j=1}^{k_{\A}}\1\{L_{i1}=j\}Z_{j}(X_i, t-T_i)$, 
when $\A=\A_d\times\A_c=\{1,\ldots,k_{\A}\}\times[0,\infty)$.
\end{itemize}

\subsubsection{Geostatistical functional data}
When observations have been made of a spatio-temporal random field, at a set of fixed known locations $x_i \in \X$, $i=1,...,n$, one often speaks of \emph{geostatistical functional data}. The class of related data types comprise a broad family of spatially dependent functional data. 
For a good account of these types of data, the reader is referred to 
\Citep{DelicadoGiraldoComas,GiraldoDelicadoMateu2010,GiraldoDelicadoMateu2011}. 

Here, given some spatial functional process 
$
\left\{Z_{x}:\ x\in \X\subseteq\R^{d}\right\},
$
we assume to observe a set of functions, or rather spatially located curves, $\left(Z_{x_{1}}(t),\ldots,Z_{x_{n}}(t)\right)$ at locations $x_i \in \X$, $i=1,...,n$, for $t \in\TT=[a,b]$, which define the set of functional observations. Each function is assumed to belong to a Hilbert space
$$
L_2(\TT)=\{f:\TT\rightarrow \R : \int_{\TT} f(t)^2 dt < \infty\}
$$
with the inner product $\langle f,g \rangle=\int_{\TT} f(t)g(t) dt$. 
Moreover, for a fixed site $x_{i}$, the observed data is assumed to follow the model
$$
Z_{x_{i}}(t) = \mu_{x_{i}}(t) + \epsilon_{x_{i}}(t),\ i=1,\ldots,n,
$$
where $\epsilon_{x_{i}}(t)$ are zero-mean residual processes and each $\mu_{x_{i}}(\cdot)$ is a mean function which summarises the main structure of $Z_{x_{i}}$.
For each $t$, we assume that the process is a second-order stationary functional random process. That formally means that the expected value $\E[Z_{x}(t)]=\mu(t)$, $t\in\TT$, $x\in\X$, and the variance $\Var(Z_{x}(t))={\sigma^{2}}(t)$, $t\in\TT$, $x\in \X$, do not depend on the spatial location. 
In addition, we have that
\begin{itemize}
\item $\Cov(Z_{x_{i}}(t),Z_{x_{j}}(t))=\mathbb{C}(h,t)$, 
where $h=\left\|x_{i}-x_{j}\right\|$,  for all $t\in\TT$ and all $x_{i}, x_{j}\in \X$.
\item $\frac{1}{2}\Var(Z_{x_{i}}(t)-Z_{x_{j}}(t))=\gamma(h,t)=\gamma_{x_i x_j}(t)$, 
where $h=\left\|x_{i}-x_{j}\right\|$, for all $t\in\TT$ and all $x_{i},x_{j}\in \X$.
\end{itemize}
Note that under the second-order stationarity assumption one may write $\frac{1}{2}\Var(Z_{x}(t)-Z_{x+h}(t))$, $h=\left\|x_{i}-x_{j}\right\|$, for $\frac{1}{2}\Var(Z_{x_{i}}(t)-Z_{x_{j}}(t))$. However, for clarity we do prefer the more general formulation. 
Since we are assuming that the mean  function is constant over $\X$, the function $\gamma(h,t)$, called the \emph{variogram} of $Z_x(t)$, can be expressed by
\begin{equation*}
\gamma(h,t)=\gamma_{x_i x_j}(t)=\frac{1}{2}\Var(Z_{x_{i}}(t)-Z_{x_{j}}(t))=\frac{1}{2}\E\left[Z_{x_{i}}(t)-Z_{x_{j}}(t)\right]^{2}.
\end{equation*}
By integrating this expression over $\TT$, using Fubini's theorem and following \Citep{GiraldoDelicadoMateu2010}, 
a measure of spatial variability is given by
\bea
\label{TraceVariogram}
\gamma(h)=\frac{1}{2}\E\left[\int_{\TT}(Z_{x_{i}}(t)-Z_{x_{j}}(t))^2 dt \right]
\eea
for $x_i, x_j \in \X$  with $h=\|x_i-x_j\|$. 
This is the so-called \emph{trace-variogram} and it is used to describe the spatial variability among functional data across an entire spatial domain. In this case, all possible location pairs are considered.

Consider now the scenario where one would perform some geostatistical analysis, such as spatio-temporal prediction \Citep{GiraldoDelicadoMateu2010}, in a spatio-temporal random field when, in addition, there is uncertainty in the monitoring locations $x_i$, $i=1,\ldots,n$. 
Note that one then instead samples the random field/spatial functional process $Z$ at locations $X_i=x_i+\varepsilon_i$, $i=1,\ldots,n$, where each $\varepsilon_i$ follows some suitable spatial distribution. 
Here the CFMPP framework is the correct one since $\{X_i\}_{i=1}^n$ constitutes a spatial point process. 
Consequently, the above geostatistical framework could be extended to incorporate such randomness in the sampling locations. 
In the deterministic case, i.e.\ when $\varepsilon_i\equiv0$, \Citep{GiraldoDelicadoMateu2010} proposed the estimator 
$\widehat{Z_{x_0}}(t)=\sum_{i=1}^{n}\lambda_i(t) Z_{X_i}(t)=\sum_{i=1}^{n}\lambda_i(t) M_i(t)$, $\lambda_i:\TT\rightarrow\R$, $i=1,\ldots,n$, for the marginal random process $\{Z_{x_0}(t)\}_{t\in\TT}$, $x_0\in\X$.
Assuming that the $X_i$'s are in fact random, following \Citep{GiraldoDelicadoMateu2010}, the associated prediction problem may be expressed as
\begin{align*}
&\min_{\lambda_1,\ldots,\lambda_n\in L_2(\TT)} \E\left[\int_{\TT}(\widehat{Z_{x_0}}(t)-Z_{x_0}(t))^2 dt \right]
=
\min_{\lambda_1,\ldots,\lambda_n\in L_2(\TT)} 
\E\left[\int_{\TT}\left(\sum_{i=1}^{n}\lambda_i(t) M_i(t) - Z_{x_0}(t)\right)^2 dt \right]
\\
&=
\min_{\lambda_1,\ldots,\lambda_n\in L_2(\TT)} 
\int_{\X^n}  \int_{\TT}
\E\left[\left.\left(\sum_{i=1}^{n}\lambda_i(t) Z_{x_i}(t) - Z_{x_0}(t)\right)^2 \right| X_1=x_1,\ldots, X_n=x_n\right] dt
\\
&\times
\frac{j_n^G(x_1,\ldots,x_n)}{n!} d(x_1,\ldots,x_n)
\end{align*}
by Fubini's theorem, 
where $j_n^G(\cdot)$ is the $n$th Janossy density of the ground process $\Psi_G=\{X_i\}_{i=1}^n$ (see Section \ref{SectionFinitePPs}). 
Hence, one obtains a geostatistical analysis based on CFMPPs.

\subsection{LISA and LISTA functions}

In the context of spatial point processes, \Citep{CollinsCressie} developed exploratory data analytic tools, in terms of Local Indicators of Spatial Association (LISA) functions based on the product density, to examine individual points in the point pattern in terms of how they relate to their neighbouring points. For each point $X_i$ of the point process/pattern we can attach to it a LISA function $M_i(h)$, $h=\|X_i-x\|\geq0$, $x\in\X$, which determines the local spatial structure associated to each event of the pattern. These functions can be regarded as functional marks \Citep{MateuLorenzoPorcu}. To perform statistical inference, which is needed for example in testing for local clustering, \Citep{CollinsCressie} developed closed form expressions of the auto-covariance and cross-covariance between any two such functions. These covariance structures are complicated to work with as they live in high-dimensional spaces.

If the ground point pattern evolves in time, i.e.\ if we have a spatio-temporal point pattern, then we can extend the ideas of LISA functions to incorporate time in their structure. In this case, local versions of spatio-temporal product densities provide the concept of LISTA surfaces \Citep{Rodriguez}. Attached to each spatio-temporal location $(X_i,T_i)$ we now have surfaces $M_i(x,t)$, $(x,t)\in\X\times\T$ (i.e.\ with dimensions space and time). When we assume that $M_i(x,t)=M_i(h)$, $h=d_{\G}((X_i,T_i),(x,t))$, these surfaces can again be regarded as functional marks. The LISTA surfaces provide information on the local spatio-temporal structure of the point pattern.

\subsection{The (stochastic) growth-interaction process}\label{SectionGI}
One of the models which has given rise to a substantial part of the ideas underlying the construction of STCFMPPs is the \emph{growth-interaction process}. It has been extensively studied in a series of papers (see e.g.\ \citep{Comas,MateuFMPP,Cronie,CronieSarkka,CronieForest,RenshawComas,RenshawComasMateu,RS1,RS2}), mainly within the forestry context. However, its representation as a functional marked point process has only been noted in \Citep{MateuFMPP,Cronie}. 

It is a STCFMPP for which the ground process $\Psi_G$ is generated by a spatial birth-death process, which has Poisson arrivals $T_i$, with intensity $\alpha>0$, and uniformly distributed spatial locations $X_i$. Furthermore, the auxiliary marks are the associated holding times $L_i$, which are independently $Exp(\mu)$-distributed, $\mu>0$, and, conditionally on the previous components, the functional marks are given by a system of ordinary differential equations,
\beann
\frac{dM_i(t)}{dt}=g(M_i(t);\theta) - \sum_{(X_j,T_j,L_j,M_j(t))\in\Psi,\ j\neq i} h((X_i,T_i,L_i,M_i(t)),(X_j,T_j,L_j,M_j(t));\theta),
\eeann
$i=1,\ldots,N$, where $t\in\supp(M_i)=[T_i,D_i)$, $D_i=(T_i+L_i)\wedge T^*$. 
Here $g(\cdot)$ represents the individual growth of the $i$th \emph{individual}, in absence of spatial interaction with other individuals, and $h((X_i,T_i,L_i,M_i(t)),(X_j,T_j,L_j,M_j(t));\theta)$ the amount of spatial interaction to which individual $i$ is subjected by individual $j$ during $[t,t+dt]$.

As can be found in the above mentioned references, the usual application of this model is the modelling of the collective development of trees in a forest stand; $X_i$ is the location of the $i$th tree, $T_i$ is its birth time, $D_i$ its death time, and $M_i(t)$ its radius (at breast height) at time $t$. 

As one may argue that this approach does not sufficiently incorporate individual growth features in the radial growth, \Citep{Cronie} suggested that a scaled white noise processes should be added to each functional mark equation, i.e.
$$
dM_i^*(t) = dM_i(t) + \sigma(M_i(t);\theta)dW_i(t),
$$
where $W_1(t),\ldots,W_N(t)$, are independent standard Brownian motions and $\sigma(\cdot)$ is some suitable diffusion coefficient. 
Here the noise would represent measurement errors and give rise to individual growth deviations. 
The resulting stochastic differential equation marked point process, the \emph{stochastic growth-interaction process}, was then studied in the simplified case where the spatial interaction is negligible, i.e.\ $h(\cdot)\equiv0$.

\subsection{Applications}\label{SectionApplications}
Besides the applications mentioned previously, we here give a list of further possible applications of (ST)CFMPPs, providing a wide scope of the current framework.

\begin{enumerate}
\item {\em Modelling nerve fibres:} 
$X_i$ gives the location of the root of the nerve. 
A mark $M_i$ (here continuous) provides the shape of the actual nerve fibre and the related auxiliary variable is given by $L_i=(L_{i1},L_{i2})\in\A_c=[0,\infty)\times[0,2\pi)$, where $\supp(M_i)=[0,L_{i1})$ and $L_{i2}$ represents a random rotation angle of $M_i$, which gives the direction of the fibre.

\item {\em Spread of pollutant:} 
$X_i$ is the pollution location, $M_i(h)$ gives us the ground concentration of the contaminant at distance $h=\|X_i-x\|$, $x\in\X$, from $X_i$.

\item {\em Modelling tumours:} 
$\X$ represents (a region in) the human body,  
$X_i$ is the location of the $i$th tumour and $M_i(t)$ its approximate radius at time $t$.

\item {\em Disease incidences in epidemics:} 
Each $M_i(t)$ is a stochastic process with piecewise constant sample paths (e.g.\ a Poisson process), which counts the number of incidences having occurred by time $t$ at epidemic centre $X_i$.  

\item {\em Population growth:} 
$X_i$ is the location of a village/town/city, $T_i$ the time point at which it was founded and $M_i(t)$ its total population at time $t$. 

\item {\em Mobile communication:} 
Letting $\T=\TT$, consider a STCFMPP $\Psi$ where each $X_i$ represents the location of a cellphone caller who makes a call at time $T_i$, which lasts for $L_i\in\A_c=[0,\infty)$ time units, i.e. the call ends at time $D_i=(T_i+L_i)\wedge T^*$.  
Then the function $M_i(t)=\1_{[B_i,D_i)}(t)$ represents the phone call in question and the total load of a server/antenna located at $s\in\X$, which has spatial reach within the region $B\subseteq\X$, $s\in B$, is $N_s(t)=\sum_{i=1}^{N}\1_B(X_i)M_i(t)$. Assuming that the server has capacity $c_s(t)$ at time $t$, it breaks down if $\sup_{t\in\TT}c_s(t)-N_s(t)\leq0$. 
Note the connection with \Citep{Baum}.

An extension here could be to let $M_i(t) = \xi_i\1_{[T_i,D_i)}(t)$ for some random quantity $\xi_i=\xi_i(X_i,T_i,L_i)$, which represents the specific load that call $i$ puts on the network.

\end{enumerate}

\section{Point process characteristics of (ST)CFMPPs}
\label{SectionPointProcessCharacteristics}

For a wide range of summary statistics, the core elements are the product densities and the intensity function(al). We here derive these for (ST)CFMPPs under a few usual assumptions. 
In addition, we define two further, highly important, building blocks for different statistics for point processes;  the Palm measures and the Papangelou conditional intensities. 

Recall that for both types of processes the mark space is given by $\M=\A\times\F$ and to provide a general notation, which may be used to describe both CFMPPs and STCFMPPs, we write 
$\Psi = \sum_{(g,m)\in\Psi}\delta_{(g,m)} = \sum_{(g,l,f)\in\Psi}\delta_{(g,l,f)}$, 
where $x=g\in\G=\X$ in the CFMPP case and $(x,t)=g\in\G=\X\times\T$ in the STCFMPP case.

Throughout, for different measures constructed, we will use the following measure extension approach. 
When some set function $\mu(A)$ is defined for the bounded Borel sets $A$ in some Borel space $(\mathcal{X},\B(\mathcal{X}))$, by assuming that $\mu(\cdot)$ is locally finite, $\mu(\cdot)$ becomes a finite measure on the ring of bounded Borel sets. 
Hereby one may extend $\mu$ to a measure on the whole $\sigma$-algebra $\B(\mathcal{X})$ (see e.g.\ \Citep[Theorem A, p. 54]{Halmos}).

\subsection{Product densities and intensity functionals}
We first consider the (factorial) moment measures and the product densities of a (ST)CFMPP $\Psi$. The construction of the product densities paves the way for the construction of certain likelihood functions and summary statistics. 

We start by defining the \emph{factorial moment measures}. 
\begin{definition}
For any $n\geq1$ and bounded 
$A_1,\ldots,A_n = (G_1\times H_1),\ldots,(G_n\times H_n)\in\B(\Y)$, define 
\begin{align*}
&\alpha^{(n)}(A_1\times\cdots\times A_n)
=
\E\left[
\sum\nolimits_{(g_1,l_1,f_1),\ldots,(g_n,l_n,f_n)\in \Psi}^{\neq}
\prod_{i=1}^{n}\1\{(g_i,l_i,f_i)\in A_i\} 
\right],
\end{align*}
where $\sum^{\neq}$ denotes a sum over distinct elements. Note that $\alpha^{(n)}$ may be extended to a measure on the $n$-fold product $\sigma$-algebra $\B(\Y)^n=\bigotimes_{i=1}^{n}\B(\Y)$, the $n$th \emph{factorial moment measure}. 
\end{definition}
Note that in the STCFMPP setting, $G_i=B_i\times C_i\in\B(\X\times\T)$ and $H_i=D_i\times E_i\in\B(\A\times\F)$, $i=1,\ldots,n$.

Recall next the reference measure $\nu$ in (\ref{ReferenceMeasure}) and assume that 
$\alpha^{(n)}\ll \nu^n$, i.e.\ that $\alpha^{(n)}$ is absolute continuous with respect to the $n$-fold product measure of the reference measure $\nu$ with itself. 
This leads to the definition of (functional) product densities.

\begin{definition}
The permutation invariant Radon-Nikodym derivatives (measurable functionals) $\rho^{(n)}$, which are defined by the integral formula 
\begin{align}
\label{Campbell}
&\E\left[
\sum\nolimits_{(g_1,l_1,f_1),\ldots,(g_n,l_n,f_n)\in\Psi}^{\neq}
h((g_1,l_1,f_1),\ldots,(g_n,l_n,f_n))
\right]
=
\\
&=
\int_{\Y^n} h((g_1,l_1,f_1),\ldots,(g_n,l_n,f_n)) \alpha^{(n)}(d(g_1,l_1,f_1)\times\cdots\times d(g_n,l_n,f_n))
\nn
\\
&=
\int_{\Y}\cdots\int_{\Y} h((g_1,l_1,f_1),\ldots,(g_n,l_n,f_n)) 
\rho^{(n)}((g_1,l_1,f_1),\ldots,(g_n,l_n,f_n)) \prod_{i=1}^{n} \nu(d(g_i,l_i,f_i))
\nn
\end{align}
for any measurable functional $h:\Y\rightarrow[0,\infty)$, are referred to as the $n$th \emph{product densities}. 
\end{definition}
Note here that (\ref{Campbell}) is referred to as the \emph{Campbell theorem}. 
Furthermore, $\rho^{(n)}$ is partly a (functional) density on $\F$, as discussed in Section \ref{SectionRefernceMeasures}.
We note that, heuristically, the interpretation of 
$\rho^{(n)}((g_1,l_1,f_1),\ldots,(g_n,l_n,f_n)) \prod_{i=1}^{n} \nu(d(g_i,l_i,f_i))$ is 
the probability of finding ground process points in the infinitesimal regions $dg_1,\ldots,dg_n\subseteq\G$, with associated marks in the infinitesimal regions $d(l_1,f_1),\ldots,d(l_n,f_n)\subseteq\A\times\F$.

Similarly to $\alpha^{(n)}$ we may also define the $n$th \emph{moment measures} 
\[
\mu^{(n)}(A_1\times\cdots\times A_n) = \E[\Psi(A_1)\cdots\Psi(A_n)], 
\quad A_1,\ldots,A_n \in\B(\Y), 
n\geq1.
\]
Furthermore, since the \emph{intensity measure} $\mu$ of a simple point process coincides with its first moment measure and its first factorial moment measure, i.e.
\bea
\label{IntensityMeasure}
\mu(A) = \E[\Psi(A)] = \alpha^{(1)}(A) = \int_{A}\rho^{(1)}(g,l,f)\nu(d(g,l,f)),
\quad A\in\B(\Y),
\eea 
we obtain a definition of intensity functionals for (ST)CFMPPs.

\begin{definition}
The \emph{intensity functional} of a (ST)CFMPP $\Psi$ is given by 
$$
\lambda(g,l,f) = \rho^{(1)}(g,l,f).
$$
\end{definition}
Note that we use the term \emph{functional} since the mapping $\lambda:\G\times\A\times\F\rightarrow[0,\infty)$ takes a c\`adl\`ag function $f\in\F$ as one of its arguments.

Turning to the ground process $\Psi_{G}$, 
through $\alpha^{(n)}$ we may define the $n$th \emph{ground factorial moment measure} $\alpha_{G}^{(n)}$ and its Radon-Nikodym derivative $\rho_{G}^{(n)}$ with respect to $\ell^n$, the $n$th \emph{ground product density}.
\begin{definition}
When the measure 
\beann
\alpha_{G}^{(n)}(G_1\times\cdots\times G_n) 
= \alpha^{(n)}((G_1\times\M)\times\cdots\times(G_n\times\M)),
\eeann
is assumed to be locally finite, it becomes the $n$th \emph{factorial moment measure of the ground process} $\Psi_{G}$. 
When $\alpha_{G}^{(n)}\ll\ell^n$, we refer to the corresponding Radon-Nikodym derivative $\rho_{G}^{(n)}$ as the $n$th \emph{ground product density}. 

The ground intensity measure is given by 
\bea
\label{GroundIntensityMeasure}
\mu_G(G)=\E[\Psi_G(G)]=\alpha_{G}^{(1)}(G)=\int_{G}\lambda_{G}(g)dg, 
\quad G\in\B(\G),
\eea
where 
$
\lambda_{G}(g) = \rho_{G}^{(1)}(g)
$
is referred to as the \emph{ground intensity function}.
\end{definition}

We note that here 
$
\alpha_{G}^{(n)}(d(g_1,\ldots,g_n)) = 
\rho_{G}^{(n)}(g_1,\ldots,g_n) dg_1\cdots dg_n
$ 
is interpreted as the infinitesimal probability of finding points of $\Psi_G$ in $dg_1,\ldots,dg_n\subseteq\G$.

\subsubsection{Product densities in terms of conditional mark distribution densities}

The following observation regarding regular probabilities on the product spaces $\B((\A\times\F)^n)$, $n\geq1$, which we state in the form of a lemma, will be exploited frequently.

\begin{lemma}
\label{RemarkDisintegration}
For any $n\geq1$, assuming that there exists some family 
\beann
P^{\M,n}=\{P_{g_1,\ldots,g_n}^{\M}(H_1\times\cdots\times H_n):(g_1,\ldots,g_n)\in\G^n, H_1\times\cdots\times H_n\in\B(\M^n)\}
\eeann
of regular probabilities, 
it follows that for any $(g_1,\ldots,g_n)\in\G^n$ and any $(D_1\times E_1),\ldots,(D_n\times E_n)\in\B(\A\times\F)$, 
\begin{align*}
&P_{g_1,\ldots,g_n}^{\M}((D_1\times E_1)\times\cdots\times(D_n\times E_n))
=
\\
&=\int_{D_1\times\cdots\times D_n}
P_{(g_1,l_1),\ldots,(g_n,l_n)}^{\F}(E_1\times\cdots\times E_n)
P_{g_1,\ldots,g_n}^{\A}(d(l_1,\ldots,l_n)),
\end{align*}
where 
$
P_{g_1,\ldots,g_n}^{\A}(\cdot\times\cdots\times\cdot) 
= P_{g_1,\ldots,g_n}^{\M}((\cdot\times\F)\times\cdots\times(\cdot\times\F))
$ 
and 
$P_{(g_1,l_1),\ldots,(g_n,l_n)}^{\F}(\cdot\times\cdots\times\cdot)$, $l_1,\ldots,l_n\in\A$,  
are families of regular probabilities on $(\A^n,\B(\A^n))$ and $(\F^n,\B(\F^n))$, respectively. 

Assume further that each $P_{g_1,\ldots,g_n}^{\M}(\cdot)$ has a density 
$f_{g_1,\ldots,g_n}^{\M}(\cdot)$ on $(\A\times\F)^n$ 
with respect to $(\nu_{\A}\otimes\nu_{\F})^n=\nu_{\A}^n\otimes\nu_{\F}^n$. 
Then it follows that $P_{g_1,\ldots,g_n}^{\A}(\cdot)\ll \nu_{\A}^n$ and 
$P_{(g_1,l_1),\ldots,(g_n,l_n)}^{\F}(\cdot)\ll \nu_{\F}^n$ for each of 
the regular probabilities on $\B(\A^n)$ and $\B(\F^n)$, 
whereby there exist associated densities 
\bea
\label{AuxMarkDensities}
f^{\A,n} &=& \{f_{g_1,\ldots,g_n}^{\A}(\cdot):g_1,\ldots,g_n\in\ \G\},
\\
\label{FunMarkDensities}
f^{\F,n} &=& 
\{f_{(g_1,l_1),\ldots,(g_n,l_n)}^{\F}(\cdot):
((g_1,l_1),\ldots,(g_n,l_n))\in(\G\times\A)^n\}
,
\eea
on $\A^n$ and $\F^n$, such that 
$$
f_{g_1,\ldots,g_n}^{\M}((l_1,f_1),\ldots,(l_n,f_n)) 
=
f_{(g_1,l_1),\ldots,(g_n,l_n)}^{\F}(f_1,\ldots,f_n)
f_{g_1,\ldots,g_n}^{\A}(l_1,\ldots,l_n)
$$
for all $(g_1,\ldots,g_n)\in\G^n$ and almost all $(l_1,f_1),\ldots,(l_n,f_n)\in\A\times\F$. 
\end{lemma}

\begin{proof}
The existence of all $P_{g_1,\ldots,g_n}^{\A}(\cdot)$ and $P_{(g_1,l_1),\ldots,(g_n,l_n)}^{\F}(\cdot)$, 
$(g_1,l_1),\ldots,(g_n,l_n)\in\G\times\A$, is a direct consequence of \Citep[Proposition A1.5.III]{DVJ1}, since $\A^n$ and $\F^n$ are csm spaces for any $n\geq1$. 

Assuming that $P_{g_1,\ldots,g_n}^{\M}(\cdot)\ll\nu_{\A}^n\otimes\nu_{\F}^n$ results in
\begin{align*}
&P_{g_1,\ldots,g_n}^{\M}((D_1\times E_1)\times\cdots\times(D_n\times E_n))
=
\\
&=\int_{(D_1\times E_1)\times\cdots\times(D_n\times E_n)}
f_{g_1,\ldots,g_n}^{\M}((l_1,f_1),\ldots,(l_n,f_n))
\prod_{i=1}^{n}\nu_{\A}(dl_i)\nu_{\F}(df_i).
\end{align*}
On the other hand, from the absolute continuity it clearly follows that $P_{g_1,\ldots,g_n}^{\A}(\cdot)\ll\nu_{\A}^n$ and $P_{(g_1,l_1),\ldots,(g_n,l_n)}^{\F}(\cdot)\ll\nu_{\F}^n$, whereby
\begin{align*}
&P_{g_1,\ldots,g_n}^{\M}((D_1\times E_1)\times\cdots\times(D_n\times E_n))
=
\\
&=\int_{(D_1\times E_1)\times\cdots\times(D_n\times E_n)}
f_{(g_1,l_1),\ldots,(g_n,l_n)}^{\F}(f_1,\ldots,f_n)
f_{g_1,\ldots,g_n}^{\A}(l_1,\ldots,l_n)
\prod_{i=1}^{n}\nu_{\A}(dl_i)\nu_{\F}(df_i).
\end{align*}
The two integrands are equal a.e..

\end{proof}

Next, through the above lemma, one finds that when the product densities exist, they may be expressed through conditional probability densities on the different mark spaces.

\begin{prop}\label{PropositionProdDens}

For any $n\geq1$, 
given that the $n$th ground factorial moment measure $\alpha_{G}^{(n)}$ exists, there exist 
regular probabilities 
\bea
P^{\M,n}&=&\{P_{g_1,\ldots,g_n}^{\M}(H):(g_1,\ldots,g_n)\in\G^n, H\in\B(\M^n)\},
\nn
\\
\label{AuxMarkProb}
P^{\A,n} &=&\{P_{g_1,\ldots,g_n}^{\A}(D):g_1,\ldots,g_n\in\ \G, D\in\B(\A^n)\}
\ll \nu_{\A}^n,
\\
\label{FunMarkProb}
P^{\F,n} &=& \{P_{(g_1,l_1),\ldots,(g_n,l_n)}^{\F}(E):(g_1,l_1),\ldots,(g_n,l_n)\in\G\times\A, E\in\B(\F^n)\}
\ll \nu_{\F}^n,
\eea
as indicated in Lemma \ref{RemarkDisintegration}. 
When $\rho^{(n)}$ exists, 
we can find conditional densities (\ref{AuxMarkDensities}) and (\ref{FunMarkDensities}) such that 
\begin{align}
\label{ProductDensityCFMPP}
&
\rho^{(n)}((x_1,l_1,f_1),\ldots,(x_n,l_n,f_n)) 
= 
\\
&=
f_{(x_1,l_1),\ldots,(x_n,l_n)}^{\F}(f_1,\ldots,f_n)
f_{x_1,\ldots,x_n}^{\A}(l_1,\ldots,l_n)
\rho_{G}^{(n)}(x_1,\ldots,x_n)
\nn
\end{align}
when $\Psi$ is a CFMPP and 
\begin{align}
\label{ProductDensitySTCFMPP}
&
\rho^{(n)}((x_1,t_1,l_1,f_1),\ldots,(x_n,t_n,l_n,f_n))
= \\
&=
f_{(x_1,t_1,l_1),\ldots,(x_n,t_n,l_n)}^{\F}(f_1,\ldots,f_n)
f_{(x_1,t_1),\ldots,(x_n,t_n)}^{\A}(l_1,\ldots,l_n)
\rho_{G}^{(n)}((x_1,t_1),\ldots,(x_n,t_n))
\nn
\end{align}
when $\Psi$ is a STCFMPP. 
Hereby the intensity functional $\lambda(g,l,f)$ turns into 
\begin{align*}
\begin{array}{ll}
\lambda(x,l,f) = 
f_{(x,l)}^{\F}(f) f_{x}^{\A}(l) \lambda_{G}(x) & \text{if }\Psi\text{ is a CFMPP,}
\\
\lambda(x,t,l,f) = 
f_{(x,t,l)}^{\F}(f) f_{(x,t)}^{\A}(l) \lambda_{G}(x,t) & \text{if }\Psi\text{ is a STCFMPP.}
\end{array}
\end{align*}

\end{prop}

\begin{proof}
For fixed $H_1,\ldots,H_n\in\B(\M)$ we have that 
$\alpha^{(n)}((\cdot\times H_1)\times\cdots\times(\cdot\times H_n))\ll\alpha_{G}^{(n)}$ and (since the underlying spaces are Borelian) a regular family $P^{\M,n}$ 
of conditional regular probabilities on $\B(\M)^n = \B(\A\times\F)^n$ exists, as indicated in Lemma \ref{RemarkDisintegration} (see e.g.\ \Citep{SKM}). An application of Lemma \ref{RemarkDisintegration} further gives us $P^{\A,n}$ and $P^{\F,n}$. 

By the existence of $\rho^{(n)}$, the underlying absolute continuity with respect to $\nu^n=(\ell\otimes\nu_{\A}\otimes\nu_{\F})^n$ implies that $\alpha_{G}^{(n)}\ll\ell^n$, whence 
\begin{align*}
&\alpha_{G}^{(n)}(G_1\times\cdots\times G_n) 
= \int_{G_1}\cdots\int_{G_n} \rho_{G}^{(n)}(g_1,\ldots,g_n) dg_1\cdots dg_n
\end{align*}
for any $G_1,\ldots,G_n\in\B(\G)$.
In addition, it follows that each member of $P^{\M,n}$ is absolute continuous with respect to $(\nu_{\A}\otimes\nu_{\F})^n$,
whereby an application of Lemma \ref{RemarkDisintegration} gives us the specific structure of the product densities. 

\end{proof}

The family $P^{\M,n}$ is often referred to as the $n$-point mark distributions. 
We see here that each $P_{(g_1,l_1),\ldots,(g_n,l_n)}^{\F}(\cdot)$ in expression (\ref{FunMarkProb}) is the distribution of some 
(mark) c\`adl\`ag stochastic process $\{(M_1(t),\ldots,M_n(t))\}_{t\in\TT}$ on $(\F^n,\B(\F^n))$, which is absolutely continuous with respect to the reference measure $\nu_{\F}^n$ (i.e.\ the distribution of an $n$-dimensional version of the reference process $X_{\F}$), with density (\ref{FunMarkDensities}). 
Furthermore, consider a STCFMPP $\Psi$ for which we want to have $\supp(M_i)=S_i$ for some $S_i=S_i(X_i,T_i,L_i)\subseteq\TT$, $i=1,\ldots,N$. 
Conditioning on $\Psi_G$, the auxiliary marks and $(X_i,T_i,L_i)=(x_i,t_i,l_i)$, $i=1,\ldots,n$, 
it then clearly follows that $f_{(x_1,t_1,l_1),\ldots,(x_n,t_n,l_n)}^{\F}(f_1,\ldots,f_n)=0$ if, for any $i=1,\ldots,n$, it holds that  
$f_i\in\F\setminus\{f\in\F : \supp(f)=S_i\}$. 
Note that often a natural choice for the supports is $S_i=[T_i,(T_i+L_i)\wedge T^*)$, $L_i\geq0$, $i=1,\ldots,N$.

\subsubsection{Pair correlation functionals}
Pair correlation functions are valuable tools for studying second order dependence properties of spatial(-temporal) point processes and, as we shall see, they play a similar role for (ST)CFMPPs.

In light of Proposition \ref{PropositionProdDens}, 
when $\lambda$ and $\rho^{(2)}$ exist, we may define the \emph{pair correlation functional} of a (ST)CFMPP $\Psi$ by 
\begin{align}
\label{PairCorrelationFunction}
g_{\Psi}((g_1,l_1,f_1),(g_2,l_2,f_2)) 
&= 
\frac{\rho^{(2)}((g_1,l_1,f_1),(g_2,l_2,f_2))}{\lambda(g_1,l_1,f_1)\lambda(g_2,l_2,f_2)}
\\
&= 
\frac{f_{(g_1,l_1),(g_2,l_2)}^{\F}(f_1,f_2)f_{g_1,g_2}^{\A}(l_1,l_2)}
{f_{(g_1,l_1)}^{\F}(f_1) f_{g_1}^{\A}(l_1)f_{(g_2,l_2)}^{\F}(f_2) f_{g_2}^{\A}(l_2)}
\frac{\rho_G^{(2)}(g_1,g_2)}{\lambda_G(g_1)\lambda_G(g_2)}
\nn
\\
&= \frac{f_{(g_1,l_1),(g_2,l_2)}^{\F}(f_1,f_2)f_{g_1,g_2}^{\A}(l_1,l_2)}
{f_{(g_1,l_1)}^{\F}(f_1) f_{g_1}^{\A}(l_1)f_{(g_2,l_2)}^{\F}(f_2) f_{g_2}^{\A}(l_2)}
g_G(g_1,g_2).
\nn
\end{align}  
We refer to $g_G(g_1,g_2)$ as the \emph{ground pair correlation function}. 
Note that when $\G=\X$, $g_G(\cdot)$ is the usual pair correlation function, as can be found in e.g.\ \Citep{BaddeleyEtAl,SKM}, and when $\G=\X\times\T$, $g_G(\cdot)$ is a spatio-temporal pair correlation function, as defined in \Citep{CronieLieshoutSTPP,GabrielDiggle,MollerGhorbani}.

\subsection{Campbell and Palm measures}
We next turn to the Palm measures of a (ST)CFMPP $\Psi$ and, as usual, they are defined via the Campbell measures. 

For any $R\in\Sigma_{\mathcal{N}_{\Y}}$ and bounded 
$A= G\times H\in\B(\Y)$, we define
\beann
\mathcal{C}^{!}(A\times R)
=
\mathcal{C}^{!}(G\times H\times R)
= 
\E\left[\sum_{(g,l,f)\in\Psi\cap(G\times H)} \1\{\Psi\setminus\{(g,l,f)\}\in R\}\right],
\eeann
which we can extend to a measure on $\B(\Y)\otimes\Sigma_{\mathcal{N}_{\Y}}$. We refer to this measure as the \emph{reduced Campbell measure} and since $\mu(A)=\mathcal{C}^{!}(A\times\mathcal{N}_{\Y})$ we see that $\mathcal{C}^{!}\ll\mu$, with Radon-Nikodym derivative $P^{!(g,l,f)}$. 
Hereby 
\beann
\mathcal{C}^{!}(A\times R)
&=& 
\int_{G\times H} \mathcal{C}^{!}(d(g,l,f)\times R)
= 
\int_{G}\int_{H} P^{!(g,l,f)}(R) \mu(d(g,l,f))
\\
&=& 
\int_{G}\int_{H} P^{!(g,l,f)}(R) \lambda(g,l,f)\nu(d(g,l,f))
\eeann 

\begin{definition}
It is possible to choose a regular version of the family 
$$
\{P^{!(g,l,f)}(R) : (g,l,f)\in\Y, R\in\Sigma_{\mathcal{N}_{\Y}}\},
$$
the \emph{reduced Palm measures}, 
such that $P^{!(g,l,f)}(R)$ is a measurable functional for fixed $R\in\Sigma_{\mathcal{N}_{\Y}}$ and a measure on $\Sigma_{\mathcal{N}_{\Y}}$ for fixed $(g,l,f)\in\Y$. 
\end{definition}
We often write $\P^{!(g,l,f)}$ for the $P^{!(g,l,f)}$-reversely induced probability measure on $(\Omega,\mathcal{F})$ and $\E^{!(g,l,f)}[\cdot]$ for the associated expectation. 
Heuristically we interpret $P^{!(g,l,f)}(R)$ as the conditional probability of the event $\{\Psi\setminus\{(g,l,f)\}\in R\}$, given $\Psi\cap\{(g,l,f)\}\neq\emptyset$. 

The (non-reduced) \emph{Palm measures} may be defined through their reduced counterparts as 
$$
P^{(g,l,f)}(R) = P^{!(g,l,f)}(\{\varphi\in\mathcal{N}_{\Y}: \varphi + \delta_{(g,l,f)}\in R\}),
\quad
R\in\Sigma_{\mathcal{N}_{\Y}},
$$
and in relation hereto, the \emph{Campbell measure} of $\Psi$ is given by 
$\mathcal{C}(A\times R)=\E[\Psi(A)\1_{R}(\Psi)]$. 
Note that when $\Psi$ is stationary, $P^{!(g,l,f)}(\cdot) = P^{!(0,l,f)}(\cdot)$ and $P^{(g,l,f)}(\cdot) = P^{(0,l,f)}(\cdot)$ for all $(g,l,f)\in\Y$.

In accordance with \Citep{MCmarked}, we also define the \emph{$H$-marked reduced Palm measure} 
\bea
\label{MarkedPalm}
P_{H}^{!g}(R) = \frac{1}{[\nu_{\A}\otimes\nu_{\F}](H)} \int_{H} P^{!(g,l,f)}(R) [\nu_{\A}\otimes\nu_{\F}](d(l,f)),
\quad H\times R\in\B(\M)\otimes\Sigma_{\mathcal{N}_{\Y}}, 
\eea
and we denote the corresponding expectation by $\E_{H}^{!g}[\cdot]$. 
Note that the choice of name for $P_{H}^{!g}(\cdot)$ may be a bit misleading since it does not necessarily define a Palm distribution in the true sense. One case, however, where this is indeed the case is when $\Psi$ has a common mark distribution (see Section \ref{SectionIndependentMarks}).

We next consider the \emph{Campbell-Mecke formula} and the \emph{reduced Campbell-Mecke formula}, which are central tools in the theory of point processes.

\begin{thm}\label{CampbellMecke}
For any measurable functional $h:\G\times\A\times\F\times \mathcal{N}_{\Y}\rightarrow[0,\infty)$, 
the \emph{reduced Campbell-Mecke formula} is given by
\begin{align*} 
&\E\left[\sum_{(g,l,f)\in\Psi} h(g,l,f,\Psi\setminus\{(g,l,f)\})\right]
=
\int_{\G\times\M}\int_{\mathcal{N}_{\Y}}
h(g,l,f,\varphi)
P^{!(g,l,f)}(d\varphi)
\mu(d(g,l,f))
\nn\\
&=
\int_{\G\times\M}
\E^{!(g,l,f)}\left[h(g,l,f,\Psi)\right]
\lambda(g,l,f)\nu(d(g,l,f))
\end{align*}
and the \emph{Campbell-Mecke formula} is given by
\begin{align*} 
&\E\left[\sum_{(g,l,f)\in\Psi} h(g,l,f,\Psi)\right]
=
\int_{\G\times\M}
\E^{(g,l,f)}\left[h(g,l,f,\Psi+\delta_{(g,l,f)})\right]
\lambda(g,l,f)\nu(d(g,l,f))
\end{align*} 
(with the left hand sides being infinite if and only if the right hand sides are infinite). 
\end{thm}

\begin{proof}
The proof is standard (see e.g.\ \Citep[Theorem 1.5]{VanLieshout}). First we show that the above expressions hold for $h(g,l,f,\varphi)=\1\{(g,l,f)\in G,\varphi\in R\}$, $G\times R\in\B(\G)\times\Sigma_{\mathcal{N}_{\Y}}$. 
The rest of the proof follows from standard arguments using linear combinations and monotone limits. 
\end{proof}

\subsection{Marked Papangelou conditional intensities}
\label{SectionPapangelou}
Another important set of point process statistics, here defined in the context of STCFMPPs, are the (marked) Papangelou kernels and the (marked) Papangelou conditional intensities. We will see that the former, in particular, will play an important role in the statistical inference framework constructed here. 

By assuming that, for fixed bounded $A=(B\times C)\times(D\times E)\in\B(\Y)$, $\mathcal{C}^{!}(A\times\cdot)$ is absolutely continuous with respect to the distribution $P(\cdot)$ of $\Psi$ on $\Sigma_{\mathcal{N}_{\Y}}$, we find that
$
\mathcal{C}^{!}(A\times R) = \int_{R}\Lambda(A;\varphi) P(d\varphi)
$ 
for some kernel $\{\Lambda(A;\varphi):A\in\B(\Y), \varphi\in\mathcal{N}_{\Y}\}$. This kernel may be extended to the $\B(\Y)$-setting, the so-called \emph{Papangelou kernel}. 
If we assume that $\Lambda(\cdot;\varphi)\ll\ell\otimes\nu_{\A}$, through Fubini's theorem we obtain 
\beann
\mathcal{C}^{!}(A\times R) 
= \int_{R}\int_{(B\times C)\times D} \lambda_E(g,l;\varphi) dg\nu_{\A}(dl) P(d\varphi)
= \int_{(B\times C)\times D} \E[\1_{R}(\Psi)\lambda_E(g,l;\Psi)] dg\nu_{\A}(dl)
\eeann
for $R\in\Sigma_{\mathcal{N}_{\Y}}$. 

\begin{definition}
We refer to $\lambda_E(g,l;\Psi)$, $E\in\B(\F)$, as the \emph{$E$-functional marked Papangelou conditional intensity}. 
\end{definition}
Here $\lambda_E(g,l;\Psi)dg\nu_{\A}(dl)$ may be viewed as the conditional probability of finding a point of $\Psi$ located in the infinitesimal set $dg\subseteq\G$ with mark in $dl\times E\in\B(\M)$, given all points of $\Psi$ located outside $dg$. 

By additionally assuming that that all functional marked Papangelou conditional intensities are absolutely continuous with respect to $\nu_{\F}$, i.e.\ assuming that $\mathcal{C}^{!}(A\times\cdot)\ll\nu$, we obtain as Radon-Nikodym derivatives the classical \emph{Papangelou conditional intensities} $\lambda(g,l,f;\Psi)$ which may also be defined,  
in integral terms, as the non-negative measurable functionals satisfying the \emph{Georgii-Nguyen-Zessin} formula, i.e.\ 
\begin{align*}
&\E\left[\sum_{(z,l,f)\in Y}h(g,l,f,\Psi\setminus\{(g,l,f)\})\right]
=
\int_{\G\times\M}\E\left[h(g,l,f,\Psi) \lambda(g,l,f;\Psi)\right] \nu(d(g,l,f))
\end{align*}
for any non-negative measurable functional $h:(\G\times\M)\times \mathcal{N}_{\Y}\rightarrow[0,\infty)$. 
Here, heuristically, $\lambda(g,l,f;\Psi)\nu(d(g,l,f))$ is interpreted as the conditional probability of finding a point of $\Psi$ in the infinitesimal region $d(g,l,f)\subseteq\G\times\M$, given the configuration elsewhere, $\Psi\cap(d(g,l,f))^c$. Note here that 
\bea
\label{RelationPalmPapangelou}
\E\left[h(g,l,f,\Psi) \lambda(g,l,f;\Psi)\right]\stackrel{a.e.}{=}\E^{!(g,l,f)}\left[h(g,l,f,\Psi)\right] \lambda(g,l,f)
\eea
and $\E[\lambda(g,l,f;\Psi)]=\lambda(g,l,f)$.

\section{Mark structures}\label{SectionMarkStructures}

We now provide some examples of explicit mark structures, which may be considered when constructing (ST)CFMPP models. Note that 
\begin{itemize}
\item multivariate (ST)CFMPPs (Sections \ref{SectionCFMPP}, \ref{SectionSTCFMPP} and \ref{SectionAuxiliaryMarkMeasure}),
\item usual marked (spatio-temporal) point processes (Section \ref{SectionMPP}),
\item spatio-temporal geostatistical marking (Section \ref{SectionGeostatMarkin}),
\end{itemize}
have already been covered previously in the text.

\subsection{Independent marks and common marginal mark distributions}\label{SectionIndependentMarks}
In light of Proposition \ref{PropositionProdDens}, whenever 
$$
P_{(g,l)}^{\F}(E) = \nu_{\F}(E), 
\quad 
E\in\B(\F), 
$$
so that $f_{(g,l)}^{\F}(\cdot)\equiv1$, we say that $\Psi$ has a \emph{common (marginal functional) mark distribution}. 
Recalling the reference process $X_{\F}$ in (\ref{ReferenceProcess}), which has $\nu_{\F}$ as distribution, we see that under a common marginal functional mark distribution, $M_i\stackrel{d}{=}X_{\F}$ for all $i=1,\ldots,N$.
Note that this is a univariate property. 
One instance where this automatically holds is when $\Psi$ is stationary. 
\begin{cor}\label{CorMarkDistrStationarity}
If $\Psi$ is stationary then all marks $(L_i,M_i)$, $i=1,\ldots,N$, have the same marginal distribution $\nu_{\A}\otimes\nu_{\F}$. 
\end{cor}

\begin{proof}
Assuming that $\Psi$ is stationary, we have that 
$\mu_{G}(G)=\lambda_{G}\ell(G)$, $\lambda_{G}>0$. 
It now follows that $\nu_{\A}\otimes\nu_{\F}$ is the uniquely determined probability measure on $(\M,\B(\M))$
which satisfies 
$\mu(A)
=\lambda_G\ell(G)[\nu_{\A}\otimes\nu_{\F}](H)
$, $A=G\times H\in B_{\Y}$ \Citep[Thm 3.5.1.]{SchneiderWeil}. 
This implies that $\lambda(g,l,f)=\lambda_G$, i.e.\ that $f_{(g,l)}^{\F}(f)\equiv1$ and $f_{g}^{\A}(l)\equiv1$.
\end{proof}

Turning next to the multivariate distributions of the functional marks, we may impose different levels of independence. 
\begin{definition}
Whenever the regular probability distributions (\ref{FunMarkProb}) satisfy 
\[
P_{(g_1,l_1),\ldots,(g_n,l_n)}^{\F}(E_1\times\cdots\times E_n) = \prod_{i=1}^{n} P_{(g_i,l_i)}^{\F}(E_i), 
\quad 
E_1,\ldots,E_n\in \B(\F),
\]
for any $n\geq1$, 
we say that $\Psi$ has \emph{independent functional marks}.

When $\Psi$ both has independent functional marks and a common marginal functional mark distribution, we say that $\Psi$ has \emph{randomly labelled functional marks}, i.e., for any $n\geq1$ we have that 
$
P_{(g_1,l_1),\ldots,(g_n,l_n)}^{\F}(E_1\times\cdots\times E_n) 
= \prod_{i=1}^{n} \nu_{\F}(E_i)
$, 
$E_1,\ldots,E_n\in \B(\F)$.

\end{definition}

Here, through Proposition \ref{PropositionProdDens}, we find that the pair correlation functional (\ref{PairCorrelationFunction}) satisfies 
\[
g_{\Psi}((g_1,l_1,f_1),(g_2,l_2,f_2)) 
= \frac{f_{g_1,g_2}^{\A}(l_1,l_2)}
{f_{g_1}^{\A}(l_1) f_{g_2}^{\A}(l_2)}
g_G(g_1,g_2).
\]
Note that when $\Psi$ has randomly labelled functional marks, $M_1,\ldots,M_N$ are independent copies of the reference process $X_{\F}$ in (\ref{ReferenceProcess}). 
Furthermore, given that $\Psi$ has independent functional marks, by additionally assuming that the regular probabilities (\ref{AuxMarkProb}) satisfy 
\[
P_{g_1,\ldots,g_n}^{\A}(D_1\times\cdots\times D_n) = \prod_{i=1}^{n} P_{g_i}^{\A}(D_i),
\quad 
D_1,\ldots,D_n\in \B(\A)
\]
for any $n\geq1$, we retrieve the classical definition of independent marking \cite[Definition 6.4.III]{DVJ1} and consequently that of random labelling.

\begin{rem}
A weaker form of independent functional marking, \emph{conditionally independent functional marking}, may be obtained by assuming that 
$$
P_{(g_1,l_1),\ldots,(g_n,l_n)}^{\F}(E_1\times\cdots\times E_n) 
= \prod_{i=1}^{n} F_{(g_1,l_1),\ldots,(g_n,l_n)}^{\F}(E_i),
\quad 
E_1,\ldots,E_n\in \B(\F),
$$
for any $n\geq1$ and some family $\{F_{(g_1,l_1),\ldots,(g_n,l_n)}^{\F}(E):(g_1,l_1),\ldots,(g_n,l_n)\in\G\times\A, E\in\B(\F)\}$ of regular probability distributions.

\end{rem}

\subsection{Functional mark reference measures and finite-dimensional distributions}\label{SectionExplicitFunctionalMeasures}
Recall from Proposition \ref{PropositionProdDens} the probability measures $P_{(g_1,l_1),\ldots,(g_n,l_n)}^{\F}(\cdot)$ 
on $(\F^n,\B(\F^n))$, $n\geq1$. 
By assigning a specific structure to each $P_{(g_1,l_1),\ldots,(g_n,l_n)}^{\F}(\cdot)$, or equivalently to $\nu_{\F}$ and $f_{(g_1,l_1),\ldots,(g_n,l_n)}^{\F}(\cdot)$, we can determine what type of functional marks we may obtain. 
So far we have not discussed specific choices for either of these components. 

To indicate a few such choices, we start here by looking at how one could obtain deterministic marks under the current setup. We then proceed to 
considering the case where $\nu_{\F}$ is given by Wiener measure (see e.g.\ \Citep[Chapter 1]{MortersPeresBMbook} or \Citep[p.\ 2]{Skorohod}) on $(\F,\B(\F))$. 
We stress that in the latter case, most of the ideas indicated may very well be applied to, say, 
Poisson random measures (see e.g.\ \Citep{Klebaner,JacodShiryaev}) or some other L\'{e}vy process/semi-martingale generated random measure on $(\F,\B(\F))$ (see e.g.\ \Citep{JacodShiryaev,Skorohod}). 
Note e.g.\ that in the Poisson case one would be able to generate multivariate functional marks given by multivariate Poisson processes, a construction similarly to \Citep{AppleTrees}. This could likely be the required setup for Examples 4 and 5 in Section \ref{SectionApplications}. 
In order to keep the discussion below fairly compact, we will always assume that we may choose a suitable filtered probability space $(\Omega,\mathcal{F},\mathcal{F}_{\T},\P)$ under which the constructions can made.

\subsubsection{Point mass reference measures and deterministic functional marks}\label{SectionDeterministicMarks}

Consider some pre-defined deterministic function $f^*:\G\times\A\times\TT\rightarrow\R$ such that all marginal functions $f_{(g,l)}^*=\{f^*(g,l,t)\}_{t\in\TT}$, $(g,l)\in\G\times\A$, are c\`adl\`ag functions. 
If we set 
\[
P_{(g_1,l_1),\ldots,(g_n,l_n)}^{\F}(E_1\times\cdots\times E_n) 
= \int_{E_1}\cdots\int_{E_n}\prod_{i=1}^{n}P_{(g_i,l_i)}^{\F}(df_i)
= \int_{E_1}\cdots\int_{E_n}\prod_{i=1}^{n}\delta_{f_{(g_i,l_i)}^*}(df_i)
\] 
for any $n\geq1$ and $E_1,\ldots,E_n\in \B(\F)$ 
or, in connection to Lemma \ref{RemarkDisintegration} and Proposition \ref{PropositionProdDens}, under some arbitrary reference measure $\nu_{\F}$,
$$
f_{(g_1,l_1),\ldots,(g_n,l_n)}^{\F}(f_1,\ldots,f_n) = 
\delta_{(f_{(g_1,l_1)}^*,\ldots,f_{(g_n,l_n)}^*)}(f_1,\ldots,f_n),
$$ 
we may create a (ST)CFMPP with deterministic marks $M_i(t)=f^*(g_i,l_i,t)$, $i=1,\ldots,N$. 
Note that this is the case when we consider e.g.\ the LISTA functions or the classical growth-interaction process.

\subsubsection{Wiener reference measures}

Since for many applications it may be desirable to let each $M_i$ follow some diffusion process for $t\in\supp(M_i)$, we here consider a setup constructed by choosing $\nu_{\F}$ as a Wiener measure (see e.g.\ \Citep[Chapter 1]{MortersPeresBMbook} or \Citep[p.\ 2]{Skorohod}). 
More specifically, let the reference process (\ref{ReferenceProcess}) be given by a standard Brownian motion $X_{\F}=W=\{W(t)\}_{t\in\TT}$. This process is generated by Wiener measure $W_{\F}$ on $C_{\TT}(\R) = \{f:\TT\rightarrow\R : f \text{ continuous}\}$. 
Noting that $C_{\TT}(\R)\in\B(\F)$, consider the reference measure $\nu_{\F}(E) = W_{\F}(E\cap C_{\TT}(\R))$, $E\in\B(\F)$, i.e.\ the measure assigning probability 0 to discontinuous sample paths and standard Brownian motion probabilities to sample paths in $C_{\TT}(\R)$. 
It now follows that 
\begin{align*}
&P_{(g_1,l_1),\ldots,(g_n,l_n)}^{\F}(E_1\times\cdots\times E_n)
= \int_{E_1\times\cdots\times E_n} 
f_{(g_1,l_1),\ldots,(g_n,l_n)}^{\F}(f_1,\ldots,f_n)
\prod_{i=1}^{n}\nu_{\F}(df_i)
\\
&=
\int_{E_1\cap C_{\TT}(\R)\times\cdots\times E_n\cap C_{\TT}(\R)}
f_{(g_1,l_1),\ldots,(g_n,l_n)}^{\F}(f_1,\ldots,f_n)
\prod_{i=1}^{n}W_{\F}(df_i)
\end{align*}
for $E_1,\ldots,E_n\in \B(\F)$, with $P_{(g_1,l_1),\ldots,(g_n,l_n)}^{\F}(\F\times\cdots\times\F)=1$.

We may next ask ourselves the adequate question how one would obtain explicit forms for the densities $f_{(g_1,l_1),\ldots,(g_n,l_n)}^{\F}(\cdot)$. To give an indication, assume that, conditionally on $\Psi_G$ and the auxiliary marks, we want to have $(M_1(t),\ldots,M_n(t))$ given by, say, an $n$-dimensional diffusion process $(Y_1(t),\ldots,Y_n(t))$, $t\in\TT$. 
Then, under certain conditions, e.g.\ the Girsanov theorem (see e.g.\ \Citep{Klebaner,JacodShiryaev}) and the Cameron-Martin theorem \Citep{MortersPeresBMbook} give rise to explicit explicit expressions for $f_{(g_1,l_1),\ldots,(g_n,l_n)}^{\F}(\cdot)$. 
Furthermore, changing the support of each $M_i$ to some interval $C_i\subseteq\TT$ can be obtained by multiplying the density by $\delta_{\Gamma_i}(f)$, where 
$\Gamma_i = \{f\in\F: \sup\{\supp(f)\}=C_i\}$, $i=1,\ldots,n$, and/or by applying time-change/stopping results to $(Y_1(t),\ldots,Y_n(t))$ before applying e.g.\ Girsanov's theorem.
We note that such a setup would be the underlying construction for the extensions discussed in Section \ref{SectionGI}.

\subsubsection{Finite-dimensional distributions of the functional marks}

As a distribution on the function space $(\F^n,\B(\F^n))$, each $P_{(x_1,t_1,l_1),\ldots,(x_n,t_n,l_n)}^{\F}(\cdot)$ is an abstract and non-tractable object, 
despite that we sometimes may be able to explicitly define its density $f_{(g_1,l_1),\ldots,(g_n,l_n)}^{\F}(\cdot)$ with respect to some reference measure $\nu_{\F}^n$. 
Hence, for all practical and mathematically explicit purposes we turn to the \emph{finite-dimensional distributions} (fidis). 
For an informative discussion on fidis for c\`adl\`ag processes, see \Citep[1.6.2]{Silvestrov}. 

Let $\Psi$ be a STCFMPP and, conditionally on $\Psi_G$ and the auxiliary marks, 
assume that we have $\{(X_i,T_i,L_i)\}_{i\in I} = \{(x_i,t_i,l_i)\}_{i\in I}$, $I=\{1,\ldots,n\}$, and let 
$$
M_I=\{M_{I}(t)\}_{t\in\TT} = \{(M_{1}(t),\ldots,M_{n}(t))
|\{(X_{j},T_{j},L_{j})=(x_j,t_j,l_j)\}_{j=1}^{n}\}_{t\in\TT}.
$$ 
It follows that $P_{(x_1,t_1,l_1),\ldots,(x_n,t_n,l_n)}^{\F}(\cdot)$, the distribution of $M_I$ on $(\F^n,\B(\F^n))$, is uniquely determined by the fidis of $M_I$ \Citep[Lemma 1.6.1.]{Silvestrov}, 
\begin{align*}
P_{M_I}
&=
\{P_{M_I(S_k)}(A) : k\geq1, S_k=\{s_1,\ldots,s_k\}\subseteq\TT, A\in\B(\R^{n\times k})\}
\\
&=
\{\P((M_{I}(s_1),\ldots,M_{I}(s_k))\in A) : k\geq1, s_1,\ldots,s_k\in\TT, A\in\B(\R^{n\times k})\}.
\end{align*}
Here we may also choose the sets $A$ as products of half-open intervals $(-\infty,u_{ij}]\subseteq\R$, $i=1,\ldots,n$, $j=1,\ldots,k$, where each $(u_{1j},\ldots,u_{nj})$, $j=1,\ldots,k$, is a continuity point of the distribution function corresponding to $P_{M_I(S_k)}(\cdot)$. 
Although we here have considered the STCFMPP case, the CFMPP case is analogous. 
We now see that, 
conditionally on $\Psi_G$ (and thereby $N$) and the auxiliary marks, 
it follows that $\{M_{i}\}_{i=1}^{N}$ is completely determined by the collection $\bigcup_{I\in\mathcal{P}_N} P_{M_I}$. 

If in addition $P_{M_I(S_k)}$ is absolutely continuous with respect to $\ell_{n}^{k}=\ell_{1}^{n k}$, the density of $P_{M_I(S_k)}$ will be denoted by 
\bea
\label{FidiDensity}
f_{(g_1,l_1),\ldots,(g_n,l_n)}^{s_1,\ldots,s_k}(u_1,\ldots,u_n) = f_{(g_1,l_1),\ldots,(g_n,l_n)}^{s_1,\ldots,s_k}
\begin{pmatrix}
u_{11} & \cdots & u_{n1} \\
\vdots & \ddots & \vdots \\
u_{1k} & \cdots & u_{nk}
\end{pmatrix},
\eea
where $u_{ij}\in\R$, $i=1,\ldots,n$, $j=1,\ldots,k$.

\subsubsection{Markovian functional marks}\label{SectionMarkovMarks}
In many cases it may be of interest to let the functional marks be given by Markov processes. This is e.g.\ the case when each mark is given by some diffusion process.

Recall that, conditionally on $\Psi$ and the auxiliary marks, 
\bea
\label{ConditionalMarkProcess}
M_I=\{M_{I}(t)\}_{t\in\TT} = \{(M_i(t))_{i\in I}\}_{t\in\TT}
\eea
for any index set $I\in\mathcal{P}_N$. 
Define $\mathcal{F}_{t}^{M_I}=\sigma\{M_{I}(s)^{-1}(A)\in\mathcal{F} : s\in\TT\cap[0,t], A\in\B(\R)^{|I|}\}$ and 
assume that the underlying filtered probability space $(\Omega,\mathcal{F},\mathcal{F}_{\T}=\{\mathcal{F}_{t}\}_{t\in\T},\P)$ satisfies $\mathcal{F}_{t}^{M_I}\subseteq\mathcal{F}_{t}$ for any $I\in\mathcal{P}_N$. 
We say that $\Psi$ has \emph{Markovian marks} if each $M_{I}$, $I\in\mathcal{P}_N$, constitutes a Markov process, i.e., for $s\leq t$, 
\beann
\P\left(M_{I}(t)\in A|\mathcal{F}_{s}\right)
=
\P\left(M_{I}(t)\in A|M_{I}(s)\right)
,
\quad A\in\B(\R)^{|I|}. 
\eeann 
Here we refer to 
$
P_{t,s}^{M_{I}}(A;M_{I}(s))
=
\P\left(M_{I}(t)\in A|M_{I}(s)\right)
$ 
as the $M_{I}$-\emph{transition probabilities} and  
when there exist \emph{transition densities} $p_{t,s}^{M_{I}}(u_t;u_s)$, $u_t,u_s\in\R^n$, with respect to $\ell_{n}$, $n=|I|$, we find that the densities (\ref{FidiDensity}) become 
\bea
\label{TransDens}
f_{(g_1,l_1),\ldots,(g_n,l_n)}^{s_1}(u_1) 
\prod_{i=2}^{k} p_{s_i,s_{i-1}}^{M_{I}}(u_i;u_{i-1})
,\quad
u_1,\ldots,u_k\in\R^n. 
\eea

\subsection{Auxiliary reference measures and multivariate (ST)CFMPPs}
\label{SectionAuxiliaryMarkMeasure}

Turning next to the auxiliary mark space $(\A,\B(\A))$ and its reference measure $\nu_{\A}$, which we have assumed to be locally finite (the local finiteness allows us to appropriately construct kernels/regular probabilities \Citep[Exercise 9.1.16]{DVJ2}), 
recall from Section \ref{SectionAuxiliaryMarks} that we either let 
$\A=\A_d=\{1,\ldots,k_{\A}\}$, 
$\A=\A_c\subseteq\R^{m_{\A}}$, or 
$\A=\A_d\times\A_c$, where $k_{\A}, m_{\A}\geq1$. 

Note that in the first case, each mark $L_i\in\A_d$, $i=1,\ldots,N$, is a discrete random variable and we may write $\Psi_i=\sum_{(g,l,f)\in\Psi\cap\G\times\{i\}\times\F}\delta_{(g,f)}=\sum_{(g,f)\in\Psi_i}\delta_{(g,f)}$ for the restriction of $\Psi$ to the auxiliary mark set $\{i\}$, $i=1,\ldots,k_{\A}$. 
For this space we employ some finite measure $\nu_{\A_d}$ as reference measure, which implies that the regular probabilities (\ref{AuxMarkProb}) take the form 
$$
P_{g_1,\ldots,g_n}^{\A}(D_1\times\cdots\times D_n)
=\sum_{D_1\cap\A_d}\cdots\sum_{D_n\cap\A_d} f_{g_1,\ldots,g_n}^{\A}(l_1,\ldots,l_n)
\nu_{\A_d}(l_1)\cdots\nu_{\A_d}(l_n),
$$
so that each $f_{g_1,\ldots,g_n}^{\A}(l_1,\ldots,l_n)$ becomes a probability mass function/discrete probability density function. 
On this space the most natural choice for $\nu_{\A_d}$ is the counting measure $\nu_{\A_d}(\cdot)=\sum_{j=1}^{k_{\A}}\delta_{j}(\cdot)$. 
At times we want to let the intensity functional of $\Psi_i$ be given by $\lambda_i(g,f)=\lambda(g,i,f)\nu_{\A}(i)$. Then, when $\nu_{\A_d}$ is the counting measure, $\nu_{\A}(i)=1$ and $\lambda_i(g,f)=f_{(g,l)}^{\F}(f) f_{g}^{\A}(i) \lambda_{G}(g)$. Note that if we do not assume that the auxiliary marks are spatially dependent, we may simply let $P_{g_1,\ldots,g_n}^{\A}(\cdot)=\nu_{\A}^n(\cdot)$, or equivalently $f_{g_1,\ldots,g_n}^{\A}(\cdot)\equiv1$, $n\geq1$.

Turning to the second alternative, each mark $L_i\in\A_c$ becomes a (possibly) continuous $m_{\A}$-dimensional random variable. 
Here some care should be taken. In most cases the natural candidate for the reference measure $\nu_{\A_c}$ would be Lebesgue measure $\ell_{m_{\A}}$ on $(\A_c,\B(\A_c))$, whereby $f_{g_1,\ldots,g_n}^{\A}(l_1,\ldots,l_n)$ would become a probability density function in the usual sense.
However, at times one must require that $\nu_{\A_c}$ is a finite measure, i.e.\ $\nu_{\A_c}(\A_c)<\infty$. This is e.g.\ the case when we treat densities of finite (ST)CFMPPs with respect to Poisson processes (see Section \ref{SectionFinitePPs}). If $\A_c$ is a bounded subset of $\R^{m_{\A}}$, Lebesgue measure will suffice and we may rescale it into the uniform probability measure $\ell_{m_{\A}}(\cdot)/\ell_{m_{\A}}(\A)$. On the other hand, when $\A_c$ is not bounded, we choose $\nu_{\A_c}$ as some probability distribution on $(\A_c,\B(\A_c))$, to which the desired distributions of the $L_i$'s are absolutely continuous. 

In the last scenario, where $\A=\A_d\times\A_c$, we simply let the reference measure be given by $\nu_{\A}(\cdot)=[\nu_{\A_d}\otimes\nu_{\A_c}](\cdot)$, the product measure of the two measures defined on the two spaces $\A_d$ and $\A_c$. 
Here each auxiliary mark has the form $L_i=(L_{i1},L_{i2})\in\A_d\times\A_c$.  
The discrete random variable $L_{i1}$ indicates which type $1,\ldots,k_{\A}$ the $i$th point belongs to,  whereas $L_{i2}$ has the purpose of, say, controlling the functional mark(s). 
Hereby, given that $(L_{i1},L_{i2}) = (l_{i1},l_{i2})$, $i=1,\ldots,n$, the functional conditional densities (\ref{FunMarkDensities}) take the form
\bea
\label{MultivariateAuxiliaryDensity}
f_{(g_1,l_{12}),\ldots,(g_n,l_{n2})}^{\F}(\cdot; l_{11},\ldots,l_{n1}),
\quad
g_1,\ldots,g_n\in\G.
\eea
Hence, given the spatial(-temporal) locations $g_1,\ldots,g_n\in\G$ and $l_{12},\ldots,l_{n2}\in\A_c$, the functional mark distributions may still vary, depending on which type each point $i=1,\ldots,n$ is assigned. 
Note that the $L_{i2}$'s may be treated as random parameter vectors which control e.g.\ the supports of the marks. 

On the other hand, it may also be the case that $\Psi_G=(\Psi_G^1,\ldots,\Psi_G^{k_{\A}})$ is a multivariate spatial point process on e.g.\ $\G=\X$, for instance a multivariate Cox process. 
As such, it may be treated as the marked point process $\Psi_G=\{(X_i,L_{i1})\}_{i=1}^{N}$, where $L_{i1}$ indicates which $\Psi_G^1,\ldots,\Psi_G^{k_{\A}}$ a point belongs to.
By additionally assigning functional marks and auxiliary marks $\{L_{i2}\}_{i=1}^{N}$ to $\Psi_G$, with the latter possibly controlling the former through (\ref{MultivariateAuxiliaryDensity}), we obtain a CFMPP $\Psi=\{(X_i,L_{i1},L_{i2},M_i)\}_{i=1}^{N}$ where the $L_{i1}$'s govern both the ground process and the functional marks. 

\begin{rem}
Since each $L_{i2}$ is a random variable, which works as a parameter in a stochastic process, the current setup of (ST)CFMPPs connects naturally to a Bayesian stochastic process framework.
\end{rem}

\subsection{Intensity-dependent marks}
\label{SectionIntensityDependentMarks}
A step forward in the marking of stationary unmarked point processes is to allow the distributions of
the marks to be dependent on the local intensity, as suggested by \Citep{HoStoyan} and \Citep{MyllymakiPenttinen} in the context of stationary log Gaussian Cox processes \Citep{MollerSyversveen,Moller}. This intensity-dependent marking assumes conditional independence to hold for the marks, given the random intensity. 
Heuristically, these models allow the marks to be large (small) in areas of low point intensity and small (large) in areas of high intensity. 
For log Gaussian Cox process, intensity-dependent marking leads to a correlation of the marks which is affected by the second-order property of the unmarked Cox process.
The setup in \Citep{MyllymakiPenttinen} developed new marking models of such a generality that not only the mean of the mark distribution but also its variance is affected by the local intensity, and these models have been employed for the marking of log Gaussian Cox processes. 
In this context it is interesting to test for mark independence and for dependence between
marks and locations \Citep{Grabarnik2011,Schlather2004}.

In the current STCFMPP context we may extend these ideas further. 
\begin{definition}
A STCFMPP $\Psi$ with ground intensity $\lambda_G(x,t)$ is said to have \emph{intensity-dependent marks} if, conditionally on $\Psi_G$ and the auxiliary marks, the functional marks are given by $M_i(t)=\lambda_G(X_i,t)$, $t\in\TT$, $i=1,\ldots,N$.
\end{definition} 

Note that this falls in the category of deterministic marks (see Section \ref{SectionDeterministicMarks}) and the corresponding point masses on $(\F^n,\B(\F^n))$, $n\geq1$, may or may not depend on the auxiliary marks (recall the discussion in Section \ref{SectionAuxiliaryMarkMeasure}).
Moreover, as in the above mentioned references, larger marks indicate where and when there is high intensity.

\section{Specific classes of (ST)CFMPPs}\label{SectionClassesSTCFMPP}

Having defined a general structure for (ST)CFMPPs, we next turn to considering different specific constructions. 
Since by the notion of a spatio-temporal point process is often meant a temporally grounded point process, we look closer at temporally grounded (ST)CFMPPs and as a result obtain a definition of functional marked conditional intensities \Citep{CoxIsham1980,DVJ1,VereJones}. 
Furthermore, we look closer at finite (ST)CFMPPs, and then in particular Markov (ST)CFMPPs \Citep{DVJ1, VanLieshout}.
However, we start by defining (spatio-temporal) c\`adl\`ag functional marked Poisson and Cox processes \Citep{DVJ1, VanLieshout, Moller, SKM}.

\subsection{Poisson processes}
Poisson processes, the most well known point process models, are the benchmark/reference models for representing lack of spatial interaction.

Given a locally finite measure $\mu$ on $\B(\Y)$, 
we let a \emph{(spatio-temporal) c\`adl\`ag functional marked ((ST)CFM) Poisson process} $\Psi$, with intensity measure $\mu$, be defined as a Poisson process on $\Y$. In other words, for any disjoint $A_1,\ldots,A_n\in\B(\Y)$, $n\geq1$, the random variables $\Psi(A_1),\ldots,\Psi(A_n)$ are independent and Poisson distributed with means $\mu(A_i)$, $i=1,\ldots,n$, provided $A_i$ is bounded. 
When $\Psi$ has an intensity functional $\lambda(\cdot)$, i.e.\ when the intensity measure in (\ref{IntensityMeasure}) satisfies $\mu(A)=\int_{A}\lambda(g,l,f)\nu(d(g,l,f))$, through Proposition \ref{PropositionProdDens} it follows that 
\[
\rho^{(n)}((g_1,l_1,f_1),\ldots,(g_n,l_n,f_n))
= \prod_{i=1}^{n} \lambda(g_i,l_i,f_i)
= \prod_{i=1}^{n} 
f_{(g_i,l_i)}^{\F}(f_i)
f_{g_i}^{\A}(l_i)
\lambda_{G}(g_i),
\]
whereby the pair correlation functional satisfies $g_{\Psi}((g_1,l_1,f_1),(g_2,l_2,f_2)) = 1$. 
We note also that through (\ref{RelationPalmPapangelou}), since for all $(g,l,f)\in\Y$ the Palm measures satisfy $P^{!(g,l,f)}(\cdot)=P(\cdot)$, its Papangelou conditional intensity satisfies $\lambda(g,l,f;\Psi)=\lambda(g,l,f)$. 
When $\Psi$ is stationary, due to Corollary \ref{CorMarkDistrStationarity}, $\Psi$ becomes randomly labelled and 
$
\rho^{(n)}((g_1,l_1,f_1),\ldots,(g_n,l_n,f_n)) = \lambda_G^n>0.
$

\subsubsection{Ground Poisson processes}
We next relax the Poisson process assumption slightly to only concern the ground process. 
More specifically, we say that a (ST)CFMPP $\Psi$ is a \emph{(ST)CFM ground Poisson process} if 
its ground process $\Psi_G$ constitutes a simple Poisson process on $\G$. 
Note that in light of Proposition \ref{PropositionProdDens}, 
$$
\rho^{(n)}((g_1,l_1,f_1),\ldots,(g_n,l_n,f_n))
=
f_{(g_1,l_1),\ldots,(g_n,l_n)}^{\F}(f_1,\ldots,f_n)
f_{g_1,\ldots,g_n}^{\A}(l_1,\ldots,l_n)
\prod_{i=1}^{n}\lambda_{G}(g_i),
$$
where $\lambda_{G}(\cdot)$ is the ground intensity function. 
Note that here $g_G(g_1,g_2)\equiv1$, whereby the pair correlation functional satisfies 
$g_{\Psi}((g_1,l_1,f_1),(g_2,l_2,f_2)) 
= f_{g_1,g_2}^{\A}(l_1,l_2)/(f_{g_1}^{\A}(l_1)f_{g_2}^{\A}(l_2))$ if $\Psi$ has independent functional marks  
and 
$g_{\Psi}((g_1,l_1,f_1),(g_2,l_2,f_2)) = 1$ if $\Psi$ is independently marked.

\subsection{Cox processes}
We here consider Cox processes (see e.g.\ \Citep[p.\ 154]{SKM}) in the current context of c\`adl\`ag functional marking. 
These are common and interesting models for spatial clustering. 
Recall from (\ref{GroundIntensityMeasure}) that $\mu_G$ is the ground intensity measure of a (ST)CFMPP. 

\begin{definition}
Given a locally finite random measure $\Lambda_G$ on $\G$, 
a (ST)CFMPP $\Psi$ is called a \emph{(spatio-temporal) c\`adl\`ag functional marked ((ST)CFM) Cox process} (directed by $\Lambda_G$) if the ground process $\Psi_G$ constitutes a $\Lambda_G$-directed Cox process on $\G$. 
In other words, conditionally on $\Lambda_G$, $\Psi_G$ is a Poisson process with $\mu_G=\Lambda_G$. 
\end{definition} 
We note that, conditionally on $\Lambda_G$, $\Psi$ becomes a (ST)CFMPP ground Poisson process. Hence, a more suitable name for $\Psi$ would possibly be \emph{(ST)CFM ground Cox process}. 
Assume next that the locally finite random measure $\Lambda_G(G)=\int_{G}\Lambda(g)dg$ is generated by an a.s.\ non-negative random field $\Lambda=\{\Lambda(g)\}_{g\in\G}$, which consequently must be a.s.\ locally integrable. 
When $\G=\X$ (CFM Cox process), we may write $\Lambda=\{\Lambda(x)\}_{x\in\X}$ and when $\G=\X\times\T$ (STCFM Cox process) we may write $\Lambda=\{\Lambda(x,t)\}_{(x,t)\in\X\times\TT}$. 
For a (ST)CFM Cox process, in light of Proposition \ref{PropositionProdDens}, the $n$th product density is given by \Citep[Chapter 6.2.]{DVJ1}
$$
\rho^{(n)}((g_1,l_1,f_1),\ldots,(g_n,l_n,f_n))
=
f_{(g_1,l_1),\ldots,(g_n,l_n)}^{\F}(f_1,\ldots,f_n)
f_{g_1,\ldots,g_n}^{\A}(l_1,\ldots,l_n)
\prod_{i=1}^{n}\E[\Lambda_{G}(g_i)],
$$
whereby $g_G(\cdot)\equiv1$ and its pair correlation functional becomes
\begin{align*}
g_{\Psi}((g_1,l_1,f_1),(g_2,l_2,f_2)) 
= \frac{f_{(g_1,l_1),(g_2,l_2)}^{\F}(f_1,f_2)f_{g_1,g_2}^{\A}(l_1,l_2)}
{f_{(g_1,l_1)}^{\F}(f_1) f_{g_1}^{\A}(l_1)f_{(g_2,l_2)}^{\F}(f_2) f_{g_2}^{\A}(l_2)}
.
\end{align*}

When $\Psi$ is a (ST)CFM Cox process with spatio-temporal geostatistical marking (recall Definition \ref{DefGeostatMarking}), i.e.\ $M_i(t)=Z_{X_i}(t)$ for some spatio-temporal random field $Z=\{Z_{x}(t)\}_{(x,t)\in\X\times\TT}$, we may connect random fields and point processes simultaneously in two different ways; the driving random field $\Lambda$ 'from underneath' 
and a random field $Z$ 'from above'. 
This structure is simplified when we consider intensity dependent marks (Section \ref{SectionIntensityDependentMarks}). In the current context this translates into the following definition.
\begin{definition}
A STCFM Cox process $\Psi$ with random intensity field $\Lambda=\{\Lambda(x,t)\}_{(x,t)\in\X\times\TT}$ is said to have \emph{intensity-dependent marks} if, conditionally on $\Psi_G$ and the random field $\Lambda$, the functional marks are given by $M_i(t)=\Lambda(X_i,t)$, $t\in\TT$, $i=1,\ldots,N$.
\end{definition}

\subsection{Temporally grounded STCFMPPs and conditional intensities}\label{SectionTemporallyGrounded} 
The product densities are extremely useful tools and e.g.\ they allow for the development of an array of different tools/statistics, useful for performing statistical inference of different kinds. 
There is, however, one case which allows one to take a step further in this development and that is when a STCFMPP $\Psi$ is temporally grounded. 
More specifically, recall from Definition \ref{DefinitionGrounding} that when a STCFMPP $\Psi$ is temporally grounded we may treat it as a temporal point process $\Psi_{\T}=\{T_i\}_{i=1}^{N}$ on $\T$ with marks $\{(X_i,L_i,M_i)\}_{i=1}^{N}$. 
This allows one to exploit the natural ordering of $\T$ and thus exploit the more general theory of temporal stochastic processes, in particular that of cumulative processes (see e.g.\ \Citep{DVJ1,DVJ2}).

It should be stressed that when $\Psi$ is a CFMPP and $\X\subseteq\R$, by construction, the ground process $\Psi_G=\{X_i\}_{i=1}^{N}\subseteq\R$ may also be treated as a temporal point process with marks $\{(L_i,M_i)\}_{i=1}^{N}$. Naturally the ideas presented below still hold under this setup.

\subsubsection{The temporal ground product densities}

We first look closer at the behaviour of $\rho_{G}^{(n)}$ when $\Psi$ is temporally grounded and for this purpose we consider $G_1,\ldots,G_n = (G_1^S\times G_1^T),\ldots,(G_n^S\times G_n^T)\in\B(\G)=\B(\X)\otimes\B(\T)$.

Since $\Psi$ is temporally grounded, $\Psi_G$ may be treated as the temporal point process $\Psi_{\T}=\{T_i\}_{i=1}^{N}$ on $\T$, with associated $\X$-valued marks $\{X_i\}_{i=1}^{N}$. 
Then, under the assumption of local finiteness, we may extend 
\beann
\alpha_{G,T}^{(n)}(G_1^T\times\cdots\times G_n^T) 
= \alpha_{G}^{(n)}((\X\times G_1^T)\times\cdots\times(\X\times G_n^T))
\eeann
to become the $n$th \emph{temporal ground factorial moment measure} of $\Psi_{\T}$. 
Consequently, when we additionally assume that $\alpha_{G,T}^{(n)} \ll \ell_{1}^n$, 
where the Radon-Nikodym derivative $\rho_{G,T}^{(n)}$ is referred to as the $n$th \emph{temporal ground product density}, we may disintegrate $\alpha_{G}^{(n)}$ with respect to 
some family 
$
\{P_{t_1,\ldots,t_n}^{G,S}(\cdot):t_1,\ldots,t_n\in\T\}
$ 
of regular conditional probability distributions on $\B(\X^n)$. 
When we further assume that $P_{t_1,\ldots,t_n}^{G,S} \ll \ell_{d}^n$, with density $f_{t_1,\ldots,t_n}^{G,S}(x_1,\ldots,x_n)$, we find that
\begin{align*}
\alpha_{G}^{(n)}(G_1\times\cdots\times G_n)
&= \int_{G_1\times\cdots\times G_n}
f_{t_1,\ldots,t_n}^{G,S}(x_1,\ldots,x_n) \rho_{G,T}^{(n)}(t_1,\ldots,t_n) \prod_{i=1}^{n}dx_i dt_i
\end{align*}
and we write $\lambda_{G,T}(t) = \rho_{G,T}^{(1)}(t)$ for the corresponding intensity.
Note that when $\rho_{G}^{(n)}$ exists, $\rho_{G}^{(n)}((x_1,t_1),\ldots,(x_n,t_n)) = f_{t_1,\ldots,t_n}^{G,S}(x_1,\ldots,x_n) \rho_{G,T}^{(n)}(t_1,\ldots,t_n)$.

\begin{rem}
When on the other hand $\Psi$ is spatially grounded, $\Psi_{G}$ may be described as a spatial point process $\Psi_{\X}=\{X_i\}_{i=1}^{N}$ with $\T$-valued temporal marks $\{T_i\}_{i=1}^{N}$. Then, under analogous assumptions and in an identical fashion, we may derive the \emph{spatial ground product density} $\rho_{G,S}^{(n)}(x_1,\ldots,x_n)$ and the corresponding family $f_{x_1,\ldots,x_n}^{G,T}(t_1,\ldots,t_n)$ of densities for which the product constitutes $\rho_{G}^{(n)}((x_1,t_1),\ldots,(x_n,t_n))$, provided it exists. However, at least when $d\geq2$ (which is a natural assumption), there is no natural ordering to be gained, whereby the temporal component is absorbed into the auxiliary mark and we end up considering the usual CFMPP case.

\end{rem}

It should be pointed out at this stage that not all STCFMPPs are temporally grounded. E.g.\ when $\G=\X\times\T=\R^2\times[0,\infty)$, in the case of a Poisson process, we must require that $\alpha_{G,T}^{(1)}(C)=\int_{C}\lambda_{G,T}(t)dt=\mu_G(\R^2\times C)<\infty$ for bounded $C\in\B(\T)$ in order for $\Psi_{\T}$ to be a well-defined point process on $\T$. When $\Psi$ is a finite STCFMPP this follows automatically.

\subsubsection{Cumulative STCFMPPs and conditional intensities}
Assume now that $\T\subseteq[0,\infty)$ and consider some suitable filtered probability space $(\Omega,\mathcal{F},\mathcal{F}_{\T}=\{\mathcal{F}_{t}\}_{t\in\T},\P)$. 
Note that in this section we essentially follow the structure and notation of \Citep[Chapter 14]{DVJ2}. 

Define  
$(\M_{\T}, \B(\M_{\T}), \nu_{\M_{\T}}) = (\X\times\A\times\F, \B(\X)\otimes\B(\A)\otimes\B(\F), \ell_{d}\otimes\nu_{\A}\otimes\nu_{\F})$ so that $(X_i,L_i,M_i)$ takes values in $(\M_{\T}, \B(\M_{\T}), \nu_{\M_{\T}})$, 
and define 
\beann
\Psi_{\T}(C,K)=\Psi(B\times C\times D\times E)
\eeann
for 
$C\in\B(\T)$ and 
$K=B\times D\times E\in\B(\M_{\T})$
which gives us the \emph{cumulative process} $\Psi_{\T}:\T\times\B(\M_{\T})\times\Omega\rightarrow[0,\infty]$, 
\bea
\label{CumulativeProcess}
\Psi_{\T}(t,K)=\Psi_{\T}([0,t],K)= \sum_{i=1}^{N}\1_{K}(X_i, L_i, M_i) \1_{(-\infty,t]}(T_i),
\quad t\in\T.
\eea
We note that this, in fact, is a cumulative process in the sense of \Citep[p.\ 378]{DVJ2} since $\Psi_{\T}(t,K)$ is monotonically increasing in $t$, for fixed $K$, and a locally finite measure on $\B(\M_{\T})$ for fixed $t$. 
Consider further the history/filtration 
\bea
\label{NaturalFiltration}
\mathcal{H}_{t} = 
\sigma(
\{\Psi_{\T}((s,t],K) : 
0<s\leq t, 
K\in\B(\M_{\T})
\}
)
\subseteq\mathcal{F}_{t}
\eea 
and note that $\Psi_{\T}(t,K)$ becomes progressively measurable with respect to $\mathcal{F}_{\T}$ \Citep[(A3.3.2)]{DVJ1} and the so-called mark-predictable $\sigma$-algebra \Citep[p.\ 379]{DVJ1} is coarser than $\mathcal{F}_{\T}$. 
Hence, since $\Psi_{\T}(t,K)$ is adapted to $\mathcal{F}_{\T}$, we have a setup well enough specified to accommodate  all purposes/constructions considered below.

Consider next the \emph{compensator} $A(t,K)$ of $\Psi_{\T}$, which is the unique (mark-)predictable cumulative process $A(t,K)$ such that, for each $K\in\B(\M_{\T})$, the process $\Psi_{\T}(t,K)-A(t,K)$ is a martingale with respect to $\mathcal{F}_{\T}$ \Citep[Definition 14.2.III.]{DVJ2}. 
Under the current setup, such an $A(t,K)$ exists uniquely and through it we may define a \emph{conditional intensity} \Citep[Definition 14.3.I.]{DVJ2}. The conditional intensity is an $\mathcal{F}_{\T}$-adapted process $\lambda^*(t,(x,l,f);\omega)$ (we emphasise this by writing out the $\omega$) which is measurable with respect to $\B(\T)\otimes\B(\M_{\T})\otimes\mathcal{F}$ and satisfying
\beann
A(t,K;\omega) = \int_{(0,t]\times K} \lambda^*(u,(x,l,f);\omega) du \nu_{\M_{\T}}(d(x,l,f))
\eeann
a.s.\ for all $t\in\T\setminus\{0\}$ and $K\in\B(\M_{\T})$. 
However, in order to explicitly define $\lambda^*$ we have to consider the Campbell measure $\mathcal{C}_{\Psi_{\T}}$ on $\T\times\M_{\T}\times\Omega$, which is defined through the relations 
\beann
\mathcal{C}_{\Psi_{\T}}(C\times K\times F) 
= \int_{F} 
\Psi_{\T}(C,K;\omega)
\P(d\omega),
\quad F\in\mathcal{F},
\eeann
for bounded $C\in\B(\T)$ and $K\in\B(\M_{\T})$. 
By assuming that $\mathcal{C}_{\Psi_{\T}}(\cdot)\ll\ell_1\otimes\nu_{\M_{\T}}\otimes\P$, we find that a predictable version of $\lambda^*$ is given by the corresponding Radon-Nikodym derivative  \Citep[Proposition 14.3.II.]{DVJ2}
\beann
\lambda^*(u,(x,l,f);\omega) = 
\frac{d\mathcal{C}_{\Psi_{\T}}(t, K; \omega)}
{
d(
\ell_1\otimes\nu_{\M_{\T}}\otimes\P
)
}
\eeann
and this version coincides (except possibly on a $[\ell_1\otimes\nu_{\M_{\T}}\otimes\P]$-null set) with any other conditional intensity defined through the integral equation above. 
Hereby the conditional intensity of the ground process $\Psi_{\T}^{G}$ is given by
\beann
\lambda_G^*(t;\omega) = \int_{\M_{\T}} \lambda^*(t,(x,l,f);\omega) \nu_{\M_{\T}}(d(x,l,f)).
\eeann
\begin{lemma}
\label{LemmaConditionalIntensity}
The conditional intensity is given by
$$
\lambda^*(t,(x,l,f);\omega) 
= f_{(x,t,l)}^{\F}(f;\omega) f_{(x,t)}^{\A}(l;\omega) f_{t}^{G,S}(x;\omega) \lambda_G^*(t;\omega),
$$
with $\E[\lambda^*(t,(x,l,f))] = \lambda(t,x,l,f) = f_{(x,t,l)}^{\F}(f) f_{(x,t)}^{\A}(l) f_{t}^{G,S}(x)\lambda_G^*(t)$, and 
\beann
\lambda_K^*(t;\omega) = \int_{K} \lambda^*(t,(x,l,f);\omega)[\ell\otimes\nu_{\A}\otimes\nu_{\F}](d(x,l,f))
\eeann
is the conditional intensity of $\Psi_{\T}(t,K)$. 
\end{lemma}

\begin{proof}
By \Citep[Proposition 14.3.II.]{DVJ2}, there exist regular probabilities $P_t^{\M_{\T}}(\cdot)$ on $\B(\M_{\T})$ which are absolutely continuous with respect to $\nu_{\M_{\T}}$ with densities $f_{\M_{\T}}(x,l,f | t;\omega)$ such that 
$\lambda^*(t,(x,l,f)) = f_{\M_{\T}}(x,l,f|t;\omega) \lambda_G^*(t;\omega)$. 
By applying Lemma \ref{RemarkDisintegration} we obtain the first statement. 
Note further that since $\E[A(t,K)]=\E[\Psi_{\T}(t,K)]$ it follows that $\E[\lambda^*(t,(x,l,f))] = \lambda(t,x,l,f)$, the intensity functional of $\Psi$. 
The last part follows from \Citep[p.\ 392]{DVJ2}. 
\end{proof}
Note here that $\lambda^*$ depends on $\Psi$ only through the history $\mathcal{H}_t$ and not the whole realisation (emphasised by $\omega$) so we consequently may write $\lambda^*(\cdot;\mathcal{H}_t)$ or simply $\lambda^*(\cdot)$, which is more common notation in point process literature. 
Provided that the limit exists, there exists a version of $\lambda^*$ such that 
\beann
\lambda^*(t,x,l,f) 
= \lim_{\Delta\downarrow0}
\frac{\E[\Psi(B_{\X}[x,\Delta]\times [t,t+\Delta]\times B_{\A}[l,\Delta]\times B_{\F}[f,\Delta]) | \mathcal{H}_t]}
{\ell_d(B_{\X}[x,\Delta]) \ell_1([t,t+\Delta]) \nu_{\A}(B_{\A}[l,\Delta]) \nu_{\F}(B_{\F}[f,\Delta])},
\eeann
where e.g.\ $B_{\F}[f,\Delta]$ denotes the closed ball which is centred around $f\in\F$ with $d_{\F}$-radius $\Delta$. 

Note also that extensions where $\T=\R$ are possible \Citep[p.\ 394]{DVJ2}. However, since in most applications it is natural to assume that the temporal starting point is 0, we do not choose to consider such a setup.

To give an example, recall Section \ref{SectionAuxiliaryMarkMeasure} and assume that $\A=\A_c=\R$. By 
\Citep[Lemma 6.4.VI.]{DVJ1} we may construct a STCFMPP $\Psi$ where $\{(T_i,L_i)\}_{i=1}^{N}$ constitutes a Compound Poisson process. This is done by letting $\Psi_{\T}$ be a Poisson process with intensity $\lambda_{G,T}(t)$ and $f_{(x,t)}^{\A}(l)$ be the density of the underlying jump size kernel $F(\cdot|\cdot)$. 
Note that in this construction we allow for the jump sizes to be location dependent.

\subsubsection{Total temporal evolution and Markovian functional marks}

Here we may extend the concept of Markovian functional marks a bit. 
Let $\T=\TT$ and let $I_t=\{i:T_i\leq t\}$, $t\in\T$, where $I_t = \bigcup_{s\leq t}I_s$. 
Letting $I=I_t$ in (\ref{ConditionalMarkProcess}), we obtain 
$
M_{I_t}(s) = (M_i(s))_{i\in I_t}, 
$ 
$s\in\T$, 
i.e.\ the multivariate conditional mark process of all point observed by time $t\in\T$. 
Note also that when we assume that $\inf\{\mathrm{supp}(M_i)\}=T_i$, then $I_t = \bigcup_{s\leq t}\{i:s\in\mathrm{supp}(M_i)\}$. 
Recall now the history $\mathcal{H}_{t}$ in (\ref{NaturalFiltration}) and define 
$\mathcal{F}_{t}^{M}=\sigma\{M_{I_t}(s)^{-1}(A)\in\mathcal{F} : s\in\T\cap[0,t], A\in\B(\R)^{|I_s|}\}$ so that, if one were to study the simultaneous temporal evolution of the points and the marks, it should hold that 
$$
\mathcal{F}_{t}^* = \mathcal{H}_{t} \bigvee \mathcal{F}_{t}^{M} \subseteq\mathcal{F}_{t},
\quad
t\in\T,
$$
for the underlying filtration $\mathcal{F}_{\T} = \{\mathcal{F}_{t}\}_{t\in\T}$.
Hereby we may study the simultaneous evolution of $\Psi_{\T}$ (or $\Psi_{\T}^G$) and $M_{I_t}$, i.e.\ the \emph{total temporal evolution}.
Note here that when
$\Psi$ has Markovian marks, it follows that 
\beann
\P\left(M_{I_t}(t)\in A|\mathcal{F}_{s}\right)
=
\P\left(M_{I_t}(t)\in A|M_{I_t}(s)\right)
\eeann 
for $s\leq t$ and $A\in\B(\R^n)$.

\subsection{Finite (ST)CFMPPs}\label{SectionFinitePPs}
Recall the definitions of finite CFMPPs and finite STCFMPPs given in Definitions \ref{DefCFMPP} and \ref{DefSTCFMPP}, respectively. 
Following \Citep[Chapter 5.3]{DVJ1}, 
we find that the distribution $P$ of a finite (ST)CFMPP $\Psi$ is completely specified by 
its  \emph{Janossy measures}, 
$
J_n(A_1\times\cdots\times A_n) 
$, 
$A_1,\ldots,A_n\in\B(\Y)$, $n\geq0$, where $\sum_{n=0}^{\infty}\frac{1}{n!}J_n(\Y^n) = 1$. We assume that these have symmetric densities with respect to the $n$-fold products $\nu^n$ of the reference measure in (\ref{ReferenceMeasure}), i.e.
\beann
J_n(A_1\times\cdots\times A_n) 
= \int_{A_1\times\cdots\times A_n} j_n((g_1,l_1,f_1),\ldots,(g_n,l_n,f_n))
\prod_{i=1}^{n}\nu(d(g_i,l_i,f_i)),
\eeann
and we refer to these densities $j_n(\cdot)$ as the \emph{Janossy densities}. 
Here $\Y^0$ denotes an ideal point such that $\Y^0\times\Y = \Y\times\Y^0=\Y$ \Citep[Proposition 5.3.II.]{DVJ1}. 
The interpretation is that 
$
J_n(d(g_1,l_1,f_1)\times\cdots\times d(g_n,l_n,f_n))
$ 
gives the probability of $\Psi$ having \emph{exactly} $n$ marked points in the infinitesimal ground-mark regions $d(g_1,l_1,f_1),\ldots,d(g_n,l_n,f_n)\subseteq\Y$ and no points anywhere else. 
We note here that by construction the ground process $\Psi_G$ is also a finite point process, on $\G$, and its Janossy densities will be denoted by 
$j_n^G$, $n\geq0$.

Recalling the interpretation of $\rho^{(n)}$, it quickly becomes clear that there is a connection between  the product densities and $j_n$. 
The fundamental difference between the two is that $j_n$ gives the infinitesimal probability of having exactly $n$ points at the specified marked locations and no points anywhere else.

\begin{lemma}\label{LemmaJanossy}
For each $n\geq1$, 
recalling the conditional densities (\ref{AuxMarkDensities}) and (\ref{FunMarkDensities}), 
the Janossy densities satisfy
\begin{align*}
&j_n((g_1,l_1,f_1),\ldots,(g_n,l_n,f_n))
=
f_{(g_1,l_1),\ldots,(g_n,l_n)}^{\F}(f_1,\ldots,f_n)
f_{g_1,\ldots,g_n}^{\A}(l_1,\ldots,l_n)
j_n^G(g_1,\ldots,g_n).
\end{align*}

\end{lemma}

\begin{proof}
We here see that the existence of all $j_n$, $n\geq1$, implies that $\Psi$ is \emph{regular} in the sense of \Citep[p.\ 247]{DVJ1}. Hence, we find that $\Psi_G$ is regular in its own right and for any $n\geq1$, given $\Psi_G = \{g_1,\ldots,g_n\}$, the conditional distribution of the marks has density $f_{g_1,\ldots,g_n}^{\M}(m_1,\ldots,m_n)$ on $\M^n = (\A\times\F)^n$, with respect to $(\nu_{\A}\otimes\nu_{F})^n$ \Citep[Proposition 7.3.I.]{DVJ1}. 
Hereby an application of Lemma \ref{RemarkDisintegration} gives the desired result.

\end{proof}

In the current setup, for any bounded $G\in\B(\G)$ and any $n\geq1$, the \emph{local Janossy measure} $J_n(\cdot|G\times\M)$, i.e.\ the Janossy measure of $\Psi\cap(G\times\M)$, has a well-defined probability density 
$
j_n((g_1,l_1,f_1),\ldots,(g_n,l_n,f_n)|G\times\M)
$ 
with respect to $\nu^n$, the \emph{local Janossy density} \Citep[p.\ 247]{DVJ1}.
We further call $\Psi$ \emph{totally finite} if $\G$ is bounded and we see that when $\Psi$ is totally finite,  $j_n^G(g_1,\ldots,g_n) = j_n^G(g_1,\ldots,g_n|\G\times\M)$ may be treated as a local Janossy density.

\begin{rem}
This construction is equivalent to specifying $\Psi$ through 
a) a family of discrete probability distributions $p_n = \P(\Psi(\Y)=n) = \P(N=n)$, $n\in\N$, and 
b) a family of symmetric probability measures 
$
\Pi_n(A_1\times\cdots\times A_n) 
= \frac{J_n(A_1\times\cdots\times A_n)}{p_n n!},
$ 
$n\geq1$ (so that $J_0(\Y^0) = p_0$) \Citep[Proposition 5.3.II.]{DVJ1}. 
Hereby the density of $\Pi_n$ with respect to $\nu^n$ is $\pi_n(\cdot)=\frac{j_n(\cdot)}{p_n n!}$ so that, conditionally on $N=n$, $\pi_n(\cdot)$ determines the multivariate distribution of the points of $\Psi$ in $\Y$. 
\end{rem}

\subsubsection{Densities with respect to Poisson processes}

Recalling Section \ref{SectionAuxiliaryMarkMeasure}, assume now that the auxiliary reference measure $\nu_{\A}$ is finite on $\B(A)$. This implies that $\nu_{\A}\otimes\nu_{\F}$ is a finite measure on $\B(\M)=\B(\A)\otimes\B(\F)$. Note that $[\nu_{\A}\otimes\nu_{\F}](\A\times\F) = \nu_{\A}(\A)\nu_{\F}(\F) = \nu_{\A}(\A)$ and if $\nu_{\A}$ is given by a probability measure, $[\nu_{\A}\otimes\nu_{\F}](\A\times\F)=1$.

Let further $\Psi^*$ be a finite (ST)CFM Poisson process with ground intensity function $\lambda_*$, finite ground intensity measure $\mu_{G}^*$, i.e.\  
$\mu_{G}^*(\G)=\int_{\G}\lambda_*(g)dg<\infty$, and intensity functional $\lambda_*(g,l,f)=f_{(g,l)}^{\F}(f)f_{g}^{\A}(l)\lambda_*(g)$. 
Denote the distribution of $\Psi^*$ on $\Sigma_{\mathcal{N}_{\Y}}$ by $P_{*}$. 

Assume next that $\Psi$ is a finite (ST)CFMPP with distribution $P$ on $\Sigma_{\mathcal{N}_{\Y}}$ and assume that $P\ll P_{*}$, where the corresponding density will be denoted by $p_{\Psi}(\varphi)$, $\varphi\in\mathcal{N}_{\Y}$. 
We may equivalently and conveniently consider the supports of $P$ and $P_{*}$, 
whereby 
we may define the density as $p_{\Psi}:\mathcal{N}^{f}\rightarrow[0,\infty)$, where $\mathcal{N}^{f}=\{\{(g_1,l_1,f_1),\ldots,(g_n,l_n,f_n)\}\subseteq\Y:n<\infty\}$ is the collection of all finite subsets of $\Y$ (see e.g.\ \Citep{DVJ2,VanLieshout, Moller}). 
Hereby (see e.g.\ \Citep[Chapter 2.3]{VanLieshout} and \Citep[Chapter 6.6.1]{Moller}) 
the density $p_{\Psi}(\cdot)$ satisfies 
$\int_{\mathcal{N}^{f}}p_{\Psi}(\y)P_*(d\y)=1$ and 
\beann
\P(\Psi\in R)
&=& \sum_{n=0}^{\infty}
\frac{\e^{-\mu_{G}^*(\G)\nu_{\A}(\A)\nu_{\F}(\M)}}{n!}
\int_{\Y^n}
\1\{\{(g_1,l_1,f_1),\ldots,(g_n,l_n,f_n)\}\in R\} \times
\\
&&\times
p_{\Psi}((g_1,l_1,f_1),\ldots,(g_n,l_n,f_n)) 
\prod_{i=1}^{n} 
f_{(g_i,l_i)}^{\F}(f_i)f_{g_i}^{\A}(l_i)\lambda_*(g_i) \nu(d(g_i,l_i,f_i))
\eeann
for $R\in\Sigma_{\mathcal{N}_{\Y}}$, 
where  
\begin{align}
\label{DensityWrtPoisson}
&p_{\Psi}((g_1,l_1,f_1),\ldots,(g_n,l_n,f_n)) 
=
\e^{\mu_{G}^*(\G)\nu_{\A}(\A)\nu_{\F}(\M)} 
j_n((g_1,l_1,f_1),\ldots,(g_n,l_n,f_n))
\\
&=
\e^{\mu_{G}^*(\G)\nu_{\A}(\A)\nu_{\F}(\M)} 
f_{(g_1,l_1),\ldots,(g_n,l_n)}^{\F}(f_1,\ldots,f_n)
f_{g_1,\ldots,g_n}^{\A}(l_1,\ldots,l_n)
j_n^G(g_1,\ldots,g_n)
\nn
\end{align}
through Lemma \ref{LemmaJanossy}. 
Note also that the densities satisfy $\E[p_{\Psi}(\Psi\cup\{(g_1,l_1,f_1),\ldots,(g_n,l_n,f_n)\})]=\rho^{(n)}((g_1,l_1,u_1),\ldots,(g_n,l_n,u_n))$ \Citep[(16)]{MollerWaagepetersen}.

\subsubsection{Markov (ST)CFMPPs}
\label{SectionMarkovCFMPPs}
For (ST)CFMPPs there are many different ways in which one can describe (local) interaction structures. 
One important specific construction in the current context, when $\Psi$ is finite with density $p_{\Psi}$, is the class of \emph{Markov (ST)CFMPPs}. 

Given a symmetric and reflexive \emph{neighbour relation} $\sim$ on $\Y$, we say that points $(g_1,l_1,f_1)\in\Y$ and $(g_2,l_2,f_2)\in\Y$ are \emph{neighbours} if $(g_1,l_1,f_1)\sim(g_2,l_2,f_2)$ and we define the \emph{neighbourhood} of $(g_1,l_1,f_1)$ as $\partial(\{(g_1,l_1,f_1)\})=\{(g,l,f)\in\Y:(g_1,l_1,f_1)\sim(g,l,f)\}$. In addition, we write $\partial_{\Psi}(\{(g_1,l_1,f_1)\})=\partial(\{(g_1,l_1,f_1)\})\cap\Psi$. 

\begin{definition}
Given a symmetric and reflexive relation $\sim$ on $\Y$, we say that a (ST)CFMPP $\Psi$ is a \emph{Markov (ST)CFMPPs} if, for all $\y$ such that $p_{\Psi}(\y)>0$, 
\begin{enumerate}
\item[a)] $p_{\Psi}(\z)>0$ for all $\z\subseteq\y$,
\item[b)] for all $(g,l,f)\in\Y$, $p_{\Psi}(\y\cup\{(g,l,f)\})/p_{\Psi}(\y)$ depends only on $(g,l,f)$ and $\partial_{\Psi}(\{(g_1,l_1,f_1)\})$.
\end{enumerate}

\end{definition}

If  $(g_1,l_1,f_1)\sim(g_2,l_2,f_2)$ for all $(g_1,l_1,f_1),(g_2,l_2,f_2)\in\y$, we say that $\y$ is a \emph{clique}. 
By writing $c_{\y}^{\sim}=\{\z\subseteq\y:\z\text{ is a clique under}\sim\}$, the Hammersley-Clifford theorem \Citep[(2.7)]{VanLieshout} states that there is a measurable interaction function $\phi:\mathcal{N}^{f}\rightarrow[0,\infty)$ such that 
$
p_{\Psi}(\y) = \prod_{\z\in c_{\y}^{\sim}} \phi(\z)
$
if and only if $\Psi$ is a Markov (ST)CFMPP. 
This is a more common way of defining a Markov process. 
We may express b) in the above definition through the Papangelou conditional intensity.
\begin{lemma}
A finite (ST)CFMPP with density $p_{\Psi}$ with respect to a Poisson process $\Psi^*$ has a Papangelou conditional intensity, which is given by $$
\lambda(g,l,f;\y)=\frac{p_{\Psi}(\y\cup\{(g,l,f)\})}{p_{\Psi}(\y)}
$$
for $(g,l,f)\notin\y\in\mathcal{N}^{f}$ and $\lambda(g,l,f;\y)=0$ for $(g,l,f)\in\y$. 
\end{lemma}

\begin{proof}
The proof exploits the Campbell-Mecke formula and the Georgii-Nguyen-Zessin formula. It follows the exact steps of  \Citep[Theorem 1.6]{VanLieshout}, with obvious modifications for the marks. 
That $\lambda(g,l,f;\y)=0$ for $(g,l,f)\in\y$ follows since $\Psi$ is simple.
\end{proof}

For convenience we often write $i\sim j$ if $(g_i,l_i,f_i),(g_j,l_j,f_j)\in\Psi$ are neighbours. 
Note here that a Markov STCFMPP $\Psi$ may possess several different types of Markovianity:
\begin{enumerate} 
\item Spatial Markovianity: e.g.\ $i\sim j$ if $d(X_i,X_j)\leq R$, for some maximum range of interaction $R>0$. 

\item Temporal Markovianity: e.g.\ $i\sim j$ if $|T_i-T_j|\leq R$, $R>0$.

\item Spatio-temporal Markovianity: e.g.\ $i\sim j$ if $(X_j,T_j)\in (X_i,T_i)+ C_{R_S}^{R_T}$, $R_S,R_T>0$ (recall the cylinder sets $C_u^v$ in (\ref{CylinderSet})). For an example, see \Citep{CronieLieshoutSTPP}.

\item Functional mark Markovianity: see Section \ref{SectionMarkovMarks}.

\item If $\Psi$ is temporally grounded, the cumulative process $\Psi_{\T}$ in (\ref{CumulativeProcess}) may be a Markov process: $\P(\Psi_{\T}(t,K)\in\cdot|\mathcal{F}_{s})=P(\Psi_{\T}(t,K)\in\cdot|\Psi_{\T}(s,K))$, $s<t$. 
\end{enumerate}

Note that for a Markov (ST)CFMPP we may define a type of {\em local independent marking}.  
For cliques $I_1= \{i_1,\ldots,i_k\}$ and $I_2=\{1,\ldots,n\}\setminus I_1=\{i_{k+1},\ldots,i_n\}$, 
conditionally on $\Psi_G$, we let the distribution $P_{g_1,\ldots,g_n}^{\M}(\cdot)$ of the marks $\{(L_i,M_i)\}_{i=1}^{n}$ on $(\M^n,\B(\M^n))$ satisfy
$$
P_{g_1,\ldots,g_n}^{\M}(d(m_1,\ldots,m_n)) 
= P_{g_{i_1},\ldots,g_{i_k}}^{\M}(d(m_{i_1},\ldots,m_{i_k}))
P_{g_{i_{k+1}},\ldots,g_{i_n}}^{\M}(d(m_{i_{k+1}},\ldots,m_{i_n})),
$$
i.e.\ the marks are independent if the points are not neighbours.

\section{Discretely sampled functional marks}\label{SectionDiscretelySampledMarks}

In many situations one might be interested in different characteristics of $\Psi$ when the functional marks $\{M_{i}\}_{i=1}^{N}$ are only observed at times $S_k=\{s_1,\ldots,s_k\}\subseteq\TT$, $k\geq1$, i.e.\ instead of whole functional marks $M_i=\{M_i(t)\}_{t\in\TT}$ we consider multivariate marks of the form $(M_i(s_1),\ldots,M_i(s_k))$, $i=1,\ldots,N$. Recall that this e.g.\ is the case when we define a classical marked spatio-temporal point process $\bar{\Psi} = \{(X_i,T_i,L_i,M_i(0)):(X_i,T_i)\in\Psi_{G}\} = \{(X_i,T_i,L_i,\xi_i):(X_i,T_i)\in\Psi_{G}\}$ through a (ST)CFMPP $\Psi$, i.e.\ $S_1=\{0\}$. 

\subsection{Functional mark sampled point process characteristics}
We shall see that the expressions below play a crucial role in the statistical inference since they form the basis for e.g.\ likelihood functions, when the mark-functions are sampled at discrete times $s_1,\ldots,s_k$.

In order to accommodate such a restriction we consider sets 
\bea
\label{FidiSets}
E_{U_i}^{S_k}=\{f\in\F: (f(s_1),\ldots,f(s_k))\in U_i\in\B(\R^k)\},
\quad i=1,\ldots,n,
\eea 
which give rise to sets  
$A_1,\ldots,A_n = (G_1\times (D_1\times E_{U_1}^{S_k})),\ldots,(G_n\times (D_n\times E_{U_n}^{S_k}))\in\B(\Y)$. 
Recalling the reference stochastic process $X_{\F}$ in (\ref{ReferenceProcess}), which has $\nu_{\F}$ as distribution on $\B(\F)$, note that 
$$
\nu_{\F}(E_{U_i}^{S_m}) = \P((X_{\F}(s_1),\ldots,X_{\F}(s_k))\in U_i).
$$
Hence, under this discrete sampling of the functional marks, through the sets $E_{U_i}^{S_k}$ most characteristics are automatically transferred over to the fidis of the functional mark processes. 
Note e.g.\ that if the reference process $X_{\F}(t)=W(t)$ is, say, a standard Brownian motion, 
then $\nu_{\F}(E_{U_i}^{S_m})$ can be calculated explicitly as a product of (Gaussian) transition probabilities.

\subsubsection{Functional mark sampled product densities}
By considering sets of the form (\ref{FidiSets}), we obtain an expression for the product densities under $S_k$-sampling, the \emph{$S_k$-mark-sampled product densities}.
\begin{lemma}\label{LemmaSampleProdDens}
Under the assumptions of Proposition \ref{PropositionProdDens} and the existence of densities (\ref{FidiDensity}), the 
$n$th \emph{$S_k$-mark-sampled product density} is given by 
\begin{align*}
&\rho_{s_1,\ldots,s_k}^{(n)}((g_1,l_1,u_1),\ldots,(g_n,l_n,u_n)) 
=
\\
&=
f_{(g_1,l_1),\ldots,(g_n,l_n)}^{s_1,\ldots,s_k}(u_1,\ldots,u_n)
f_{g_1,\ldots,g_n}^{\A}(l_1,\ldots,l_n) 
\rho_{G}^{(n)}(g_1,\ldots,g_n),
\end{align*}
where $u_1,\ldots,u_n\in\R^k$.  

\end{lemma}

\begin{proof}
We may define the factorial moment measure with respect to $S_k$ and $U_1,\ldots,U_n\in\B(\R^k)$,
\begin{align*}
\alpha_{S_k}^{(n)}((G_1\times D_1\times U_1)\times\dots\times(G_n\times D_n\times U_n))
=
\alpha^{(n)}(A_1\times\dots\times A_n),
\end{align*}
where $A_1,\ldots,A_n = (G_1\times (D_1\times E_{U_1}^{S_k})),\ldots,(G_n\times (D_n\times E_{U_n}^{S_k}))\in\B(\Y)$. 
We note that by setting each $U_i=\R^k$, we obtain $\alpha_{S_k}^{(n)}=\alpha_{G}^{(n)}$, hence $\alpha_{S_k}^{(n)}\ll\alpha_{G}^{(n)}$. 
The family of regular probabilities on $\B(\A)\otimes\B(\R^{n\times k})$, which are the corresponding Radon-Nikodym derivatives, are clearly absolutely continuous with respect to the probability measures $P^{\A,n}$ in (\ref{AuxMarkProb}). Hereby, since $\rho_{G}^{(n)}$ and 
densities (\ref{AuxMarkDensities}) and (\ref{FidiDensity}) 
exist, 
\begin{align*}
&\alpha_{S_k}^{(n)}((G_1\times D_1\times U_1)\times\dots\times(G_n\times D_n\times U_n))
= 
\\
&=\int_{G_1\times D_1\times U_1}\dots\int_{G_n\times D_n\times U_n} 
f_{(g_1,l_1),\ldots,(g_n,l_n)}^{s_1,\ldots,s_k}(u_1,\ldots,u_n)
f_{g_1,\ldots,g_n}^{\A}(l_1,\ldots,l_n) 
\\
&\times
\rho_{G}^{(n)}(g_1,\ldots,g_n) 
\prod_{i=1}^{n}\nu_{\A}(dl_i) dg_i du_i
\end{align*}
and similarly, by directly assuming that $\alpha_{S_k}^{(n)}\ll\ell\otimes\nu_{\A}\otimes\ell_{n k}$, we find that 
\begin{align*}
&\alpha_{S_k}^{(n)}((G_1\times D_1\times U_1)\times\dots\times(G_n\times D_n\times U_n))
= 
\\
&=\int_{G_1\times D_1\times U_1}\dots\int_{G_n\times D_n\times U_n} 
\rho_{s_1,\ldots,s_k}^{(n)}((g_1,l_1,u_1),\ldots,(g_n,l_n,u_n)) 
\prod_{i=1}^{n}dg_i \nu_{\A}(dl_i) du_i
,
\end{align*}
which implies that the two integrands are equal a.e..
\end{proof}

As a consequence, the pair correlation functional is given by
\beann
g_{\Psi}((g_1,l_1,f_1),(g_1,l_2,f_2))
=
\frac{
f_{(g_1,l_1),(g_1,l_2)}^{s_1,\ldots,s_k}(u_1,u_2)
f_{g_1,g_2}^{\A}(l_1,l_2) 
}
{
f_{(g_1,l_1)}^{s_1,\ldots,s_k}(u_1) f_{(g_2,l_2)}^{s_1,\ldots,s_k}(u_2)
f_{g_1}^{\A}(l_1) f_{g_2}^{\A}(l_2) 
}
g_{G}(g_1,g_2).
\eeann

\subsubsection{Functional mark sampled conditional intensities}
Turning next to conditional intensities of Section \ref{SectionTemporallyGrounded}, let $U=U_1\in\B(\R^k)$ in expression (\ref{FidiSets}), recall Lemma \ref{LemmaConditionalIntensity} and define 
\beann
\lambda_{S_k}^*(t,(x,l,U);\omega) 
&=&
\int_{E_{U}^{S_k}} \lambda^*(t,(x,l,f);\omega) \nu_{\F}(df)
\\
&=&
f_{(x,t)}^{\A}(l;\omega) f_{t}^{G,S}(x;\omega) \lambda_G^*(t;\omega)
\int_{E_{U}^{S_k}} 
f_{(x,t,l)}^{\F}(f;\omega) 
 \nu_{\F}(df)
 \\
&=&
f_{(x,t)}^{\A}(l;\omega) f_{t}^{G,S}(x;\omega) \lambda_G^*(t;\omega)
P_{S_k}^{*}(U|x,t,l),
\eeann
which is well-defined since $\lambda_K^*(t;\omega)$, $K\in\B(\M_{\T})$, exists. 
If additionally the probability measure $P_{S_k}^{*}(\cdot|x,t,l)$ is absolutely continuous with respect to $\ell_{k}(\cdot)$, it follows that
$$
\lambda_{S_k}^*(t,(x,l,U);\omega) 
= 
\int_{U} 
f_{(x,t,l)}^{s_1,\ldots,s_m}(u;\omega) f_{(x,t)}^{\A}(l;\omega) f_{t}^{G,S}(x;\omega) \lambda_G^*(t;\omega)
du.
$$
We refer to the integrand 
\bea
\label{SampleConditionalIntensity}
\lambda_{S_k}^*(t,(x,l,u);\omega)
= f_{(x,t,l)}^{s_1,\ldots,s_m}(u;\omega) f_{(x,t)}^{\A}(l;\omega) f_{t}^{G,S}(x;\omega) \lambda_G^*(t;\omega)
\eea
as the \emph{$S_k$-mark-sampled conditional intensity}.

\subsubsection{Functional mark sampled Papangelou conditional intensities}
Recall the $E$-functional marked Papangelou conditional intensities $\lambda_{E}(g,l;\Psi)$ in Section \ref{SectionPapangelou}. 
In a similar fashion, one may define the marked \emph{$S_k$-mark-sampled Papangelou kernel}
\beann
\lambda^{S_k}(g,l,U;\Psi) = \lambda_{E_{U}^{S_k}}(g,l;\Psi), 
\ U\in\B(\R^k),
\eeann
through the set $E_{U}^{S_k}\in\B(\F)$. 
Note that $\lambda^{S_k}(g,l,U;\Psi)$ may be viewed as the infinitesimal conditional probability of finding a point of $\Psi$ in $dg\subseteq\G$ with auxiliary mark in $dl\subseteq\A$ and functional mark $M_i$ such that $(M_i(s_1),\ldots,M_i(s_k))\in U$, given all points of $\Psi$ located outside $dg$. 
Conditionally on the $i$th point of $\Psi$ having spatial(-temporal) location $g$ and auxiliary mark $l$, and conditionally on all points $j\neq i$ of $\Psi$, 
assume that $\lambda^{S_k}(g,l,\cdot;\Psi)$ is absolutely continuous with respect to the distribution of $(M_i(s_1),\ldots,M_i(s_k))$ and assume that this distribution, in turn, is absolutely continuous with respect to $\ell_{k}$. Then the corresponding density (see expression (\ref{FidiDensity})) is given by $f_{(g,l)}^{s_1,\ldots,s_m}(u)$, $u\in\R^m$, and consequently
\begin{align}
\label{SamplePapangelou}
\lambda^{S_m}(g,l,U;\Psi) 
= \int_{U} \lambda^{S_m}(g,l,du;\Psi)
= \int_{U} f_{(g,l)}^{s_1,\ldots,s_m}(u) \lambda^{S_m}(g,l,u;\Psi) du,
\end{align}
where we note that $f_{(g,l)}^{s_1,\ldots,s_m}(u) \lambda^{S_m}(g,l,u;\Psi)du$ may be interpreted as $\lambda^{S_m}(g,l,U;\Psi)$ with $U$ given by the infinitesimal region $du\subseteq\R^k$.

\subsubsection{Functional mark sampled Janossy densities}
A further important entity for the likelihood inference is the set of Janossy densities. 
By combining Lemma \ref{LemmaJanossy} and Lemma \ref{LemmaSampleProdDens}, we find that the $n$th \emph{$S_k$-sampled Janossy density} is given by
\begin{align}
\label{SampleJanossy}
&j_n^{S_k}((g_1,l_1,u_1),\ldots,(g_n,l_n,u_n))
=\\
&=
f_{(g_1,l_1),\ldots,(g_n,l_n)}^{s_1,\ldots,s_k}(u_1,\ldots,u_n)
f_{g_1,\ldots,g_n}^{\A}(l_1,\ldots,l_n)
j_n^G(g_1,\ldots,g_n), \nn
\end{align}
where $u_i=(u_{i1},\ldots,u_{ik})\in\R^k$, $i=1,\ldots,n$. 
In addition, from (\ref{DensityWrtPoisson}) we obtain that under the $S_k$ sampling the density of a finite (ST)CFMPP with respect to a Poisson process is given by
\bea
\label{SampleDensityWrtPoisson}
p_{\Psi}^{S_k}((g_1,l_1,u_1),\ldots,(g_n,l_n,u_n)) 
=
\e^{\mu_{G}^*(\G)\nu_{\A}(\A)\nu_{\F}(\M)} j_n^{S_k}((g_1,l_1,u_1),\ldots,(g_n,l_n,u_n)).
\eea

\subsection{Statistical inference}

We now turn to the statistical inference for (ST)CFMPPs with discretely sampled functional marks. 
Specifically, we consider a STCFMPP $\Psi=\{(X_i,T_i,L_i,M_i)\}_{i=1}^{N}$ with distribution $P_{\theta_0}$, which belongs to some parametric family of models $\{P_{\theta}:\theta\in\Theta\}$, and we here indicate how a few estimation schemes may be defined to obtain an estimate $\hat{\theta}$ of $\theta_0$.

Note here that when we observe all of $\Psi$, but with functional marks sampled according to $S_k=\{s_1,\ldots,s_k\}\subseteq\TT$, and wish to find an estimate $\hat{\theta}$, we essentially consider the usual marked spatio-temporal point process setup, where the marked point process under consideration is given by 
$\bar{\Psi} = \{((X_i,T_i),(L_i,M_i(s_1),\ldots,M_i(s_k)))\}_{i=1}^{N}$. 
Assume that we observe $\bar{\Psi}$ as $\{((x_i,t_i),(l_i,u_i))\}_{i=1}^{n}$, $u_i=(u_{i1},\ldots,u_{ik})$, within some compact space-time region $W_S\times W_T\subseteq\G$. Given the imposed probabilistic structures of the ground process and the marking, it may be possible to derive e.g.\ a (pseudo)likelihood function. 
The literature on the subject is vast and good accounts can be found in e.g.\ \Citep{DVJ1,DVJ2,Diggle2003,Handbook,Illian,VanLieshout,Moller}. 
A quick indication of how one could proceed is as follows.

\begin{itemize}
\item When $\Psi$ is assumed to be spatially grounded (see Section \ref{SectionTemporallyGrounded}), by consulting e.g.\ \Citep{DVJ1} and recalling (\ref{SampleConditionalIntensity}), we find that the log-likelihood function is given by
\beann
\log L(\theta) 
&=& 
\sum_{i=1}^{n} \log\lambda_{S_k}^*(t_i,(x_i,l_i,u_i))
-
\int_{ W_T} \int_{W_S} 
\int_{\A} \int_{\R^k}
\lambda_{S_k}^*(t,(x,l,u))
du \nu_{\A}(dl)
dx dt 
\\
&=& \sum_{i=1}^{n} 
\log f_{(x_i,t_i,l_i)}^{s_1,\ldots,s_m}(u_i) 
+ \log f_{(x_i,t_i)}^{\A}(l_i) + \log f_{t_i}^{G,S}(x_i) + \log \lambda_G^*(t_i)
\\
&&
-
\int_{ W_T} \int_{W_S} 
f_{t}^{G,S}(x)
\lambda_G^*(t)
dx dt 
.
\eeann
When we maximise $\log L(\theta)$ with respect to $\theta\in\Theta$, we obtain a maximum likelihood estimate $\hat{\theta}$. 
This is probably the most well-known likelihood estimation procedure for point processes.

\item When $\Psi$ is finite (and not temporally grounded), 
following e.g.\ \Citep[Chapter 3.7]{VanLieshout}, through expressions (\ref{SampleJanossy}) and (\ref{SampleDensityWrtPoisson}), we obtain the likelihood function as
\beann
L(\theta) 
&=&
p_{\Psi}^{S_k}((x_1,t_1,l_1,u_1),\ldots,(x_n,t_n,l_n,u_n))
\\
&\propto&
j_N^{S_k}((x_1,t_1,l_1,u_1),\ldots,(x_n,t_n,l_n,u_n))
\\
&=&
f_{(x_1,t_1,l_1),\ldots,(x_n,t_n,l_n)}^{s_1,\ldots,s_k}(u_1,\ldots,u_n)
f_{(x_1,t_1),\ldots,(x_n,t_n)}^{\A}(l_1,\ldots,l_n)
j_n^G((x_1,t_1),\ldots,(x_n,t_n)).
\eeann
By maximising $L(\theta)$ or $\log L(\theta)$ with respect to $\theta\in\Theta$, we obtain a maximum likelihood estimate $\hat{\theta}$. 

\item When $\Psi$ is a Markov STCFMPP as defined in Section \ref{SectionMarkovCFMPPs}, it is completely specified by its Papangelou conditional intensity. 
Hereby, through e.g.\ \Citep[Chapter 3.8]{VanLieshout}, by recalling expression (\ref{SamplePapangelou}), we find that the pseudo-likelihood function is given by 
\begin{align*}
&PL(\theta) 
=
\exp\left\{
-\int_{W_S\times W_T\times\A}
\lambda^{S_m}((x,t),l,\R^k;\{(x_i,t_i,l_i,u_i)\}_{i=1}^{n}) 
[\ell\otimes\nu_{\A}](d(x,t,l))
\right\}
\\
&\times
\prod_{i=1}^{n}
f_{(x_i,t_i,l_i)}^{s_1,\ldots,s_m}(u_i)
\lambda^{S_m}\left(x_i,t_i,l_i,u_i;\{(x_i,t_i,l_i,u_i)\}_{i=1}^{n}\right)
\end{align*}
and upon maximising $PL(\theta)$ or $\log PL(\theta)$ with respect to $\theta\in\Theta$, we obtain a maximum pseudo-likelihood estimate $\hat{\theta}$.

\item
For a finite $\Psi$, 
assume that we are able to estimate the parameters of the ground process and the auxiliary marks in some suitable way. 
Conditionally on $\{(x_i,t_i,l_i)\}_{i=1}^{n}$, 
given some parametric family $\{f(t;\theta,n)=(f_1(t;\theta),\ldots,f_n(t;\theta)):t\in\TT, \theta\in\Theta\}\subseteq\F^n$, $n\geq1$, of $L^2$-functions, when we 
sample the marks according to $S_k$, 
with a slight abuse of notation, 
we obtain the least-squares estimator 
\beann
\hat{\theta} 
&=& 
\argmin_{\theta\in\Theta} \int_{\TT}\sum_{j=1}^{k}\delta_{s_j}(t)\left[M_I(t|\Psi_G,\{L_i\}_{i=1}^{N}) - f(t;\theta,n)\right]^2 dt
\\
&\approx& 
\argmin_{\theta\in\Theta} \int_{\TT}\sum_{j=1}^{k}\delta_{s_j}(t)\left[M_I(t|\{(x_i,t_i,l_i)\}_{i=1}^{n}) - f(t;\theta,n)\right]^2 dt
\\
&=& \argmin_{\theta\in\Theta} \sum_{j=1}^{k} 
\sum_{i=1}^{n}
\left[
(u_{ij} - f_i(s_j;\theta,n)
\right]^2,
\eeann
where we recall $M_I$, $I=\{1,\ldots,n\}$, from (\ref{ConditionalMarkProcess}). The second (approximate) equality follows since the observed marks are likely influenced by points falling outside $W_S\times W_T$. Note hereby that this approach does not account for edge effects. To correct for edge effect in a setting as the one above, \Citep{CronieSarkka} suggested a few different approaches. The general idea is to, successively and conditionally on $\{((x_i,t_i),(l_i,u_i))\}_{i=1}^{n}$, simulate realisations on a torus $\G$, with $W_S\times W_T\subseteq\G$, from the parametric model under consideration and then use all simulated marked points falling in $\G\setminus(W_S\times W_T)$ as the missing data in $\G\setminus(W_S\times W_T)$. Assuming that there are $n^*\geq0$ simulated points falling in $\G\setminus(W_S\times W_T)$, the estimator above is adjusted by considering some parametric family $\{f(t;\theta,n+n^*):t\in\TT, \theta\in\Theta\}\subseteq\F^{n+n^*}$ of functions and the estimator $\hat{\theta}$ becomes exact.

\end{itemize}

Note above that if the $M_i$'s are Markov processes with existing transition densities, as described in Section \ref{SectionMarkovMarks}, then $f_{(x_1,t_1,l_1),\ldots,(x_n,t_n,l_n)}^{s_1,\ldots,s_k}(u_1,\ldots,u_n)$, $n\geq1$, are given by products of the transition densities found in expression (\ref{TransDens}).

\subsection{Observable processes, thinning and parameter estimation}
The above sampling structure, with spatio-temporally continuously sampled points and discretely sampled marks, may be reasonable in certain situations. However, one may argue that if the functional marks are sampled discretely, also the spatio-temporal ground process and possibly the auxiliary marks should be sampled according to whether we are able to observe them at the sample times $S_k=\{s_1,\ldots,s_k\}$.
Below we discus such a setup and indicate how the associated statistical inference could be performed.

In order to provide a mathematical structure which reflects such a sampling scheme, 
for a STCFMPP $\Psi$ 
consider the (modified) evaluation functional 
$$
\pi_t^{\F} [\varphi] = \{(x,t,l,f(t)) : (x,t,l,f)\in\varphi\},
\ \varphi\in\mathcal{N}_{\Y}, t\in\TT.
$$
When considering certain types of real data, such as in the sampling described above, it is useful to consider the \emph{observable process} 
\beann
\Psi_{O}^{\Y}(t) &=& \pi_t^{\F} [\Psi\cap(\X\times\T\times\A\times\{f\in\F:t\in\supp(f)\})]
\\
&=& \pi_t^{\F} [\{(X_i,T_i,L_i, M_i)\in\Psi : t\in \supp\{M_i\}\}]
\\
&=& \{(X_i,T_i,L_i,M_i(t)):(X_i,T_i,L_i, M_i)\in\Psi, t\in \supp\{M_i\}\}.
\eeann
Note that this corresponds to \emph{what we observe} at time $t$, i.e.\ 
we treat a point of $\Psi$ as observable/present at time $t$ only if its functional mark is non-zero. 
The definition is analogous for CFMPPs.

Consider next the scenario where we sample the observable process at times $s_1,\ldots,s_k$. 
A typical data set which corresponds to this type of sampling could be e.g.\ a forest stand, where each $X_i$ represents the location of a tree, each $T_i$ its birth time, and each $(M_i(s_1),\ldots,M_i(s_k))$ the sizes (e.g.\ radius or height) of the $i$th tree, measured at $s_1,\ldots,s_m$. 
We note that what we in fact are observing is a location-dependent thinning of $\Psi$, i.e.\ for some measurable function $p:\Y\rightarrow[0,1]$, each point $(x,t,l,f)$ of $\Psi$ is retained with the probability probability $p(x,t,l,f)$. 
The product densities of a thinned STCFMPP \Citep{DVJ2} are given by
\begin{align*}
&\rho_{\rm th}^{(n)}((x_1,t_1,l_1,f_1),\ldots,(x_n,t_n,l_n,f_n)) = 
\\
&=\rho^{(n)}((x_1,t_1,l_1,f_1),\ldots,(x_n,t_n,l_n,f_n)) \prod_{i=1}^{n} p(x_i,t_i,l_i,f_i).
\end{align*}
The correct choice of retention probability function to reflect the nature of the $S_k$-evaluated observable process is 
$$
p(x,t,l,f) = p(x,t,l,f; S_k) = \1\{S_k\cap\supp(f)\neq\emptyset\},
\quad S_k=\{s_1,\ldots,s_k\},
$$
Note that if the support of the $i$th functional mark of $\Psi$ is given by $\supp(M_i)=[t,t+l)$ when $(X_i,T_i,L_i,M_i)=(x,t,l,f)$, we set $\supp(f)=[t,t+l)$. 
Hence, we remove all unobserved points which we do not observe at any of the sample times $s_1,\ldots,s_k$. 

Turning now to the characteristics corresponding to $(\Psi_{O}^{\Y}(s_1),\ldots,\Psi_{O}^{\Y}(s_1))$, consider first 
the $S_k$-marked-sampled product densities in Lemma \ref{LemmaSampleProdDens}, which are needed since we also sample the functional marks discretely. We find that their thinned versions are given by 
\begin{align*}
&\rho_{S_k {\rm th}}^{(n)}(
(x_1,t_1,l_1,u_1),\ldots,(x_n,t_n,l_n,u_n)
)
=
\\
&=
f_{(x_1,t_1,l_1),\ldots,(x_n,t_n,l_n)}^{s_1,\ldots,s_k}(u_1,\ldots,u_n)
f_{(x_1,t_1),\ldots,(x_n,t_n)}^{\A}(l_1,\ldots,l_n) 
\rho_{G}^{(n)}((x_1,t_1),\ldots,(x_n,t_n)) 
\\
&\times
\prod_{i=1}^{n} \1\{S_k\cap\supp(f_i)\neq\emptyset\}. 
\end{align*}
Under the assumptions of Proposition \ref{PropositionProdDens}, when the product densities of all orders $n\geq1$ exist, 
through \Citep[Lemma 5.4.III.]{DVJ1} we find that
\begin{align*}
&j_n^{S_k {\rm th}}((x_1,t_1,l_1,u_1),\ldots,(x_n,t_n,l_n,u_n))
=
\\
&=\sum_{j=0}^{\infty}\frac{(-1)^j}{j!}
\int_{\X\times\T\times\A\times\R^k}\cdots\int_{\X\times\T\times\A\times\R^k} 
\times 
\\
&\times 
\rho_{S_k {\rm th}}^{(n+j)}((x_1,t_1,l_1,u_1),\ldots,(x_n,t_i,l_n,u_n),y_1,\ldots,y_j)
\prod_{i=1}^{j} [\ell_d\otimes\ell_1\otimes\nu_{\A}\otimes\ell_k](dy_i)
\\
&=
\prod_{i=1}^{n} \1\{S_k\cap\supp(f_i)\neq\emptyset\}
f_{(x_1,t_1,l_1),\ldots,(x_n,t_n,l_n)}^{s_1,\ldots,s_k}(u_1,\ldots,u_n)
f_{(x_1,t_1),\ldots,(x_n,t_n)}^{\A}(l_1,\ldots,l_n) 
\times 
\\
&\times 
\sum_{j=0}^{\infty}\frac{(-1)^j}{j!}
\int_{\G}\cdots\int_{\G} 
\rho_{G}^{(n+j)}((x_1,t_1),\ldots,(x_n,t_n),g_1,\ldots,g_j) 
\prod_{i=1}^{j} dg_i
\\
&=
f_{(x_1,t_1,l_1),\ldots,(x_n,t_n,l_n)}^{s_1,\ldots,s_k}(u_1,\ldots,u_n)
f_{(x_1,t_1),\ldots,(x_n,t_n)}^{\A}(l_1,\ldots,l_n) 
j_n^G(g_1,\ldots,g_n)
\times 
\\
&\times 
\prod_{i=1}^{n} \1\{S_k\cap\supp(f_i)\neq\emptyset\}
.
\end{align*}
Hence, the thinned $S_k$-sampled Papangelou conditional intensity $\lambda^{S_m}(x_i,t_i,l_i,u_i;\Psi)$ is given by $\1\{S_k\cap\supp(f_i)\neq\emptyset\}\lambda^{S_m}(x_i,t_i,l_i,u_i;\Psi)$.

By further integrating out over $\A$ and/or $\T$ in the expressions above, we obtain equivalences for 
\beann
\Psi_{O}^{\X\times\F}(t) 
&=& 
\mathrm{proj}_{\X\times\F}(\Psi_{O}^{\Y}(t))
=\{(X_i,M_i(t)):(X_i,T_i,L_i,M_i(t))\in\Psi_{O}^{\Y}(t)\}
\\
&=& \{(X_i,M_i(t)):(X_i,T_i,L_i, M_i)\in\Psi, t\in \supp\{M_i\}\},
\nn
\\
\Psi_{O}^{\X\times\A\times\F}(t) 
&=& 
\mathrm{proj}_{\X\times\A\times\F}(\Psi_{O}^{\Y}(t))
=\{(X_i,L_i,M_i(t)):(X_i,T_i,L_i, M_i)\in\Psi, t\in \supp\{M_i\}\},
\nn
\eeann
where e.g.\ $\mathrm{proj}_{\X\times\F}(A)$ denotes the projection of $A\subseteq\X\times\T\times\A\times\F$ onto $\X\times\F$. 
Note that such a setup might be more realistic in most cases, since e.g.\ in a forest stand we do not actually observe the birth times $T_i$, $i=1,\ldots,N$, of the trees.

\section*{Acknowledgements}
The authors are truly grateful to N.M.M.\ van Lieshout (CWI, The Netherlands) for feedback, ideas and proofreading. The authors are also grateful to to Aila S\"arkk\"a (Chalmers University of Technology, Sweden) and Eric Renshaw (University of Strathclyde, U.K.) for ideas, discussions and for introducing us to concepts underlying the field. 
This research was supported by the Netherlands Organisation for Scientific Research NWO (613.000.809) and the Spanish Ministry of Education and Science (NTN2010-14961).

\bibliographystyle{plainnat}
\bibliography{references}

\end{document}